\documentclass[11pt]{amsart}

\usepackage[T1]{fontenc}
\usepackage{amsmath,amssymb,amsthm,mathtools,mathrsfs}
\usepackage[margin=1in]{geometry}
\usepackage{microtype}
\usepackage{booktabs}
\usepackage{enumitem}
\usepackage{tikz}
\usepackage{hyperref}

\hypersetup{
  colorlinks=true,
  linkcolor=blue,
  citecolor=blue,
  urlcolor=blue,
  pdftitle={Intertwining the line-bundle and Grauert-tube Hardy quantizations of the round two-sphere},
  pdfauthor={Hy P.G. Lam},
  pdfsubject={Intertwining line-bundle and Grauert-tube Hardy quantizations on the unit cosphere bundle of the round two-sphere},
  pdfkeywords={Hardy space, unit cosphere bundle, Grauert tube, adapted complex structure, Szego kernel, intertwining operator, Legendre polynomial, principal angles, Fredholm determinant, Bargmann--Fock space}
}

\newcommand{\SO}{\mathrm{SO}}
\newcommand{\SU}{\mathrm{SU}}
\newcommand{\U}{\mathrm{U}}
\newcommand{\CP}{\mathbb{CP}}
\newcommand{\C}{\mathbb{C}}
\newcommand{\R}{\mathbb{R}}

\newcommand{\N}{\mathbb{N}}
\newcommand{\Sph}{\mathbb{S}}
\newcommand{\ii}{i}
\newcommand{\Sp}{\mathrm{Sp}}

\newcommand{\dd}{\mathrm{d}}
\newcommand{\Tr}{\operatorname{Tr}}
\newcommand{\Leg}{\mathrm{P}}
\newcommand{\Coth}{\operatorname{coth}}
\newcommand{\Csch}{\operatorname{csch}}
\newcommand{\Span}{\operatorname{span}}
\newcommand{\Id}{\mathrm{Id}}
\newcommand{\Li}{\operatorname{Li}}
\newcommand{\BF}{\mathrm{BF}}
\newcommand{\ip}[2]{\langle #1,#2\rangle}
\newcommand{\norm}[1]{\lVert #1\rVert}
\newcommand{\abs}[1]{\lvert #1\rvert}
\newcommand{\diag}{\operatorname{diag}}
\newcommand{\ket}[1]{\lvert #1\rangle}
\newcommand{\bra}[1]{\langle #1\rvert}

\theoremstyle{plain}
\newtheorem{theorem}{Theorem}[section]
\newtheorem{proposition}[theorem]{Proposition}
\newtheorem{lemma}[theorem]{Lemma}
\newtheorem{corollary}[theorem]{Corollary}

\theoremstyle{definition}
\newtheorem{definition}[theorem]{Definition}

\theoremstyle{remark}
\newtheorem{remark}[theorem]{Remark}

\title[Intertwining two Hardy quantizations]{Intertwining the line-bundle and Grauert-tube Hardy quantizations of the round two-sphere}
\author[Hy P.G. Lam]{Hy P.G. Lam}
\address{Department of Mathematical Sciences, Worcester Polytechnic Institute, Worcester, MA 01609}
\email{hlam@wpi.edu}
\address{Department of Mathematics, Northwestern University, Evanston, IL 60208}
\email{hylam.math@gmail.com, hylam2023@u.northwestern.edu}
\date{August 14, 2026}

\makeatletter
\renewcommand{\l@subsection}{\@tocline{2}{0pt}{2.5pc}{3pc}{}}
\renewcommand{\l@subsubsection}{\@tocline{3}{0pt}{4.5pc}{4pc}{}}
\makeatother

\begin{document}

\begin{abstract}
We compare two natural Hardy quantizations carried by the same unit cosphere bundle of the round two-sphere.  Through the identification $\Sph^2\simeq\CP^1$, the unit cosphere bundle is the unit circle bundle of $K_{\CP^1}\simeq\mathcal O(-2)$, the dual of the positive anticanonical line bundle $\mathcal O(2)$.  Through the real-analytic round metric, the imaginary-time exponential map identifies the same cosphere bundle with every Grauert-tube boundary $\partial M_\tau$.  The line-bundle Hardy space is organized by holomorphic sections on $\CP^1$ and the vertical circle action.  The Grauert-tube Hardy space is defined by the adapted complex structure, whose Reeb flow is the geodesic flow.  After the invariant measures are normalized, both are naturally realized in $L^2(\SO(3))$.  Each is multiplicity-free under the left $\SO(3)$ action and selects one line in each Peter--Weyl multiplicity space.  We compute the normalized overlap of these lines in closed form,
\[
  \bigl|\ip{\widehat w_\ell(\tau)}{\widehat u_\ell}\bigr|
  =\frac{\sqrt{\binom{2\ell}{\ell}}}{2^\ell}
    \frac{\sinh^\ell\!\tau}{\sqrt{\Leg_\ell(\cosh 2\tau)}}.
\]
The normalized kernel vectors define an explicit equivariant unitary between the two Hardy spaces, and their squared overlaps are the eigenvalues of the positive trace-class operator $\mathsf T_\tau=\Pi_h\Pi_\tau\Pi_h$.  We derive four exact trace series with the Hardy projectors inserted and show that their sum differs from the full flat trace by a distribution whose Abel regularization has cubic growth at every geodesic period.  The eigenvalues of $\mathsf T_\tau$ also give a genus-zero Fredholm determinant of order zero.  Truncating the complete large-$\ell$ expansion at any fixed order gives a finite polylogarithmic expression whose possible singularities occur where $r(\tau)^{2s}=1$. A Bargmann--Fock calculation gives the constant
\[
  C_{\BF}(\tau)=(2\sinh 2\tau)^{1/4}e^{-\tau/2}
\]
as the Gaussian matrix coefficient of the metaplectic operator of the linear symplectic map that compares the two contact planes.  This constant is the $\ell$-independent prefactor in the large-$\ell$ asymptotic of the overlap.
\end{abstract}

\maketitle

\section*{Acknowledgments}
This work began during the author's doctoral studies at Northwestern University and continued at Worcester Polytechnic Institute.  The author is especially grateful to the late Steve Zelditch for many conversations at Northwestern and for his work on Szeg\H{o} kernels and Grauert tubes, which shaped the questions considered here.  The author also thanks Robert Chang, Michael Geis, Abraham Rabinowitz, John Toth, and Jared Wunsch for helpful discussions.

\tableofcontents

\section{Introduction}

\subsection{Why the two Hardy quantizations are naturally paired}

Let $(M,g)$ be a compact oriented real-analytic Riemannian surface.  Its unit cosphere bundle $S^*M$ has two natural realizations as the boundary of a complex surface.

The first comes from the complex structure of the base.  The orientation and conformal class of $g$ make $M$ a Riemann surface.  Let $K_M=T^{*(1,0)}M$ be its canonical line bundle, equipped with the Hermitian metric induced by $g$.  For a real covector $\xi$, write
\[
  \xi^{1,0}=\frac12\bigl(\xi-\ii\,\xi\circ J\bigr).
\]
Taking the normalized $(1,0)$-part gives a fiberwise diffeomorphism
\begin{equation}
\label{eq:cosphere-canonical-bundle}
  \iota_{\mathrm{lb}}\colon S^*M\longrightarrow S(K_M),
  \qquad
  \iota_{\mathrm{lb}}(x,\xi)
  =\left(x,\frac{\xi^{1,0}}{\norm{\xi^{1,0}}_g}\right),
\end{equation}
where $S(K_M)$ is the unit circle bundle of $K_M$.  Thus $S^*M$ is naturally the boundary of the unit disk bundle in the holomorphic cotangent line.  If $K_M^{-1}$ is positive, this is the standard strictly pseudoconvex boundary used in positive line-bundle quantization, and its Hardy space has the Fourier decomposition
\begin{equation}
\label{eq:line-bundle-hardy-general}
  H^2\bigl(S(K_M)\bigr)
  \simeq
  \widehat\bigoplus_{\ell\ge0}H^0\bigl(M,K_M^{-\ell}\bigr).
\end{equation}
The circle action on $S(K_M)$ is the vertical rotation in the cotangent fibers \cite{ZelditchSzego,ShiffmanZelditch}.

The second realization comes from the real-analytic metric.  For every admissible Grauert radius $\tau>0$, the imaginary-time exponential map
\begin{equation}
\label{eq:cosphere-grauert-boundary}
  \iota_\tau\colon S^*M\longrightarrow \partial M_\tau,
  \qquad
  \iota_\tau(x,\xi)
  =\exp_x^{\C}\!\bigl(\ii\tau\xi^\sharp\bigr),
\end{equation}
identifies the unit cosphere bundle with the boundary of the radius-$\tau$ Grauert tube in the complexification of $M$ \cite{GuilleminStenzel,LempertSzoke,ZelditchGrauert}.  Pulling back the boundary CR structure gives another strictly pseudoconvex CR structure on $S^*M$.  Its Reeb flow is a constant reparametrization of the geodesic flow, with the constant determined by $\tau$.

For the round sphere, $M\simeq\CP^1$ and
\[
  K_M\simeq\mathcal O(-2),
  \qquad
  K_M^{-1}\simeq\mathcal O(2).
\]
The line-bundle realization therefore gives
\[
  Y=S\bigl(\mathcal O(-2)\bigr)\simeq S^*\Sph^2,
\]
and its Hardy space assembles the spaces $H^0(\CP^1,\mathcal O(2\ell))$.  The Grauert realization identifies every $\partial M_\tau$ with the same unit cosphere bundle.  Since
\begin{equation}
\label{eq:common-cosphere-identifications}
  Y\simeq S^*\Sph^2\simeq S^1\Sph^2
  \simeq \partial M_\tau\simeq\SO(3),
\end{equation}
a normalization of the invariant measures places both Hardy spaces in the same Hilbert space $L^2(\SO(3))$.

This is why the two quantizations belong in the same problem.  They quantize the same classical phase space, but they retain different parts of its geometry.  The line-bundle CR structure comes from the holomorphic cotangent line and its Chern connection, and its Reeb flow is the vertical circle action.  The Grauert-tube CR structure comes from the full real-analytic metric, and its Reeb flow is the geodesic flow.  They are complementary in geometric origin, not complementary subspaces of $L^2(\SO(3))$.  The equivariant unitary constructed below intertwines the common left $\SO(3)$-action.  It does not identify the two Reeb flows.

The round sphere is the basic homogeneous case in which this comparison can be carried out exactly.  Both Hardy spaces are left-invariant and multiplicity-free.  In the Peter--Weyl block $V_\ell\otimes V_\ell^*$, each has the form $V_\ell\otimes L_\ell$ for a line $L_\ell\subset V_\ell^*$.  The principal angle between the line-bundle and Grauert lines measures the relative position of the two Hardy structures at level $\ell$.  Its cosine is the singular value of the product of the two Szeg\H{o} projectors on that block.  The resulting overlap sequence determines the equivariant intertwiner, the spectrum of $\Pi_h\Pi_\tau\Pi_h$, and the trace, determinant, zeta, and Bargmann--Fock calculations developed below.  On a general real-analytic surface, the relative position of the two Hardy spaces is difficult to describe.  The exact formula obtained here gives a complete model calculation for that broader problem.

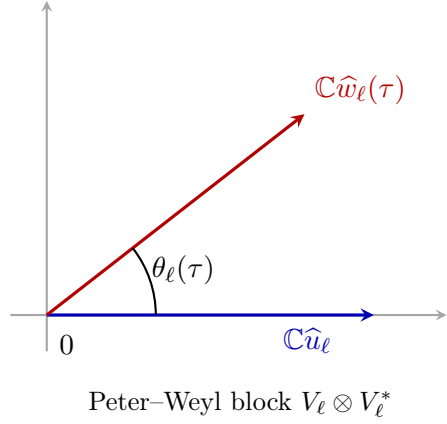
\begin{figure}[ht]
\centering
\begin{tikzpicture}[scale=1.6, >=stealth]
  \draw[->, thick, gray!70] (-0.3,0) -- (3.3,0);
  \draw[->, thick, gray!70] (0,-0.3) -- (0,2.6);
  \draw[->, very thick, blue!70!black] (0,0) -- (2.7,0);
  \node[blue!70!black, below] at (2.15,-0.02) {$\C\widehat u_\ell$};
  \draw[->, very thick, red!70!black] (0,0) -- ({2.7*cos(38)},{2.7*sin(38)});
  \node[red!70!black, above right] at ({2.7*cos(38)},{2.7*sin(38)}) {$\C\widehat w_\ell(\tau)$};
  \draw[thick] (0.9,0) arc[start angle=0, end angle=38, radius=0.9];
  \node at ({1.15*cos(19)},{1.15*sin(19)}) {~ $\theta_\ell(\tau)$};
  \node[below right] at (0.02,-0.08) {$0$};
  \node[anchor=north] at (1.6,-0.55) {\small Peter--Weyl block $V_\ell\otimes V_\ell^*$};
\end{tikzpicture}
\caption{The line-bundle, or Hopf, Hardy space and the Grauert-tube Hardy space select the lines $\C\widehat u_\ell$ and $\C\widehat w_\ell(\tau)$ in the multiplicity factor $V_\ell^*$.  Formula~\eqref{eq:overlap-formula} computes their angle.}
\label{fig:spectral-angle}
\end{figure}

\subsection{Main results}

\begin{theorem}[CR kernel lines and their overlap]
\label{thm:main}
Fix $\tau>0$.  For every $\ell\in\N_0$, the Hopf and Grauert Hardy summands determine lines $\C u_\ell$ and $\C w_\ell(\tau)$ in $V_\ell^*$.  Using the self-duality $V_\ell^*\simeq V_\ell$, we realize the multiplicity factor as the Borel--Weil space $\mathrm{Sym}^{2\ell}\C^2$.  The two lines are then represented by
\[
  u_\ell=z_0^{2\ell},
  \qquad
  w_\ell(\tau)=\bigl(z_0^2+z_1^2+2\Coth\tau\,z_0z_1\bigr)^\ell.
\]
Consequently
\[
  H^2(Y)\simeq\widehat\bigoplus_{\ell\ge0}V_\ell\otimes\C u_\ell,
  \qquad
  H^2(\partial M_\tau)\simeq\widehat\bigoplus_{\ell\ge0}V_\ell\otimes\C w_\ell(\tau).
\]
If the invariant inner product is normalized by $\norm{z_0^{2\ell}}=1$, then
\begin{equation}
\label{eq:overlap-formula}
  \bigl|\ip{\widehat w_\ell(\tau)}{\widehat u_\ell}\bigr|
  =\frac{\sqrt{\binom{2\ell}{\ell}}}{2^\ell}
    \frac{\sinh^\ell\!\tau}{\sqrt{\Leg_\ell(\cosh 2\tau)}}.
\end{equation}
\end{theorem}

We write
\[
  c_\ell(\tau)=\bigl|\ip{\widehat w_\ell(\tau)}{\widehat u_\ell}\bigr|,
  \qquad
  s_\ell(\tau)=c_\ell(\tau)^2,
  \qquad
  \theta_\ell(\tau)=\arccos c_\ell(\tau).
\]
The number $\theta_\ell(\tau)$ is the Fubini--Study distance between the two multiplicity lines.

\begin{theorem}[Equivariant unitary intertwiner]
\label{thm:intertwiner}
There is a unique $\SO(3)$-equivariant unitary
\[
  \mathcal T_\tau\colon H^2(\partial M_\tau)\longrightarrow H^2(Y)
\]
that sends $v\otimes\widehat w_\ell(\tau)$ to $v\otimes\widehat u_\ell$ for every $v\in V_\ell$.  Its extension by zero to $L^2(\SO(3))$ is a partial isometry and has block formula
\[
  \mathcal T_\tau\big|_{V_\ell\otimes V_\ell^*}
  =\Id_{V_\ell}\otimes\ket{\widehat u_\ell}\bra{\widehat w_\ell(\tau)}.
\]
\end{theorem}

The positive operator $\mathsf T_\tau=\Pi_h\Pi_\tau\Pi_h$ has eigenvalue $s_\ell(\tau)$ with multiplicity $2\ell+1$.  For fixed $\tau$ these eigenvalues decrease exponentially with $\ell$, while for fixed $\ell\ge1$ they increase strictly with $\tau$.  They define an order-zero Fredholm determinant and the zeta series
\[
  Z_{\mathsf T}(s,\tau)=\sum_{\ell\ge0}(2\ell+1)s_\ell(\tau)^s,
  \qquad \Re s>0.
\]
The exact trace formulas in Section~\ref{sec:trace-formulas} show that the sum of the four traces obtained by inserting the Hardy projectors does not agree with the full flat trace modulo a smooth function.  At every geodesic period, the Abel regularization of their difference has a cubic leading term.

\subsection{Large-\texorpdfstring{$\ell$}{ell} asymptotics}
The exact overlap has the asymptotic form
\begin{equation}
\label{eq:intro-factorization}
  c_\ell(\tau)
  =C_{\BF}(\tau)\,r(\tau)^\ell
   \bigl(1+O_\tau(\ell^{-1})\bigr),
\end{equation}
where
\[
  C_{\BF}(\tau)=(2\sinh 2\tau)^{1/4}e^{-\tau/2},
  \qquad
  r(\tau)=\frac{1-e^{-2\tau}}2.
\]
Section~\ref{sec:BF-leading} computes $C_{\BF}(\tau)$ independently as a metaplectic Gaussian matrix coefficient in the linear Heisenberg model.  This is a local calculation and does not require a global Fourier integral description of $\Pi_h\Pi_\tau$.  The coefficient agrees with the one obtained from the exact representation-theoretic formula.

The same expansion enters the determinant and the zeta series.  Truncating it after finitely many powers of $\ell^{-1}$ gives a finite combination of polylogarithms near the set $r(\tau)^{2s}=1$.  This calculation does not continue the full zeta function across $\Re s=0$.

\subsection{Relation to earlier work}
The adapted complex structure on a Grauert tube was constructed by Guillemin and Stenzel and independently by Lempert and Sz\H{o}ke \cite{GuilleminStenzel,LempertSzoke}.  The relation between Szeg\H{o} kernels on Grauert-tube boundaries and analytic continuation of eigenfunctions is developed by Zelditch \cite{ZelditchComplexZeros,ZelditchGrauert}.  The line-bundle Hardy space is the standard circle-bundle realization used in geometric quantization \cite{ZelditchSzego,ShiffmanZelditch,BBS,Berman,MaMarinescu}.

Chang and Rabinowitz study high-frequency Szeg\H{o} kernels and Husimi distributions on a single Grauert-tube boundary \cite{ChangRabinowitz}.  Here both Hardy spaces are placed in the same $L^2(\SO(3))$, and we compute their exact overlap in each multiplicity space rather than the asymptotics of either projector separately.

The geodesic Koopman operator and its flat trace also appear in the author's earlier work.  On compact hyperbolic surfaces, the Koopman representation has an explicit spectral decomposition from which the flat-trace distribution can be recovered \cite{LamHyperbolicFlatTrace}.  Sunada-type constructions give non-isometric examples with the same flat-trace data \cite{LamGuilleminRuelle}, while the first variation of the flat trace recovers marked lengths and gives local rigidity along smooth negatively curved deformations \cite{LamFirstVariation}.  Those results vary the metric or compare different metrics.  Here the metric and geodesic flow are fixed, while the CR polarization changes.  The sequence $s_\ell(\tau)$ records the angle between the two Hardy lines in each Peter--Weyl multiplicity space.

The Bargmann--Fock calculation is related to the Toeplitz calculus and to Fourier integral operators with complex phase developed by Boutet de Monvel, Guillemin, Sj\"ostrand, Melin, Borthwick, and Uribe \cite{BdMG,BdMS,MelinSjostrand,BU}.  Only the linear Heisenberg model is used here, and the coefficient reduces to a Gaussian integral.

\subsection{Comparison with heat-kernel transforms}
Hall's Segal--Bargmann transform for compact Lie groups and Stenzel's transform for compact symmetric spaces are unitary maps obtained from heat evolution followed by analytic continuation \cite{Hall,HallGQ,Stenzel,KirwinMouraoNunes}.  Their target spaces carry heat-kernel measures on complexifications or holomorphic $L^2$ structures determined by the symmetric-space geometry.  In the setting considered here, both Hardy spaces already lie in $L^2(\SO(3))$, and the relevant quantity in each Peter--Weyl block is the angle between two lines in $V_\ell^*$ rather than a heat-semigroup multiplier.

\subsection{Organization}
Sections~2--6 place the two Hardy spaces in a common Peter--Weyl decomposition and compute the exact overlap.  Sections~7--9 study the unitary intertwiner, the operator $\Pi_h\Pi_\tau\Pi_h$, and the trace formulas with the Hardy projectors.  Sections~10 and~11 treat the Fredholm determinant and the zeta function associated with $\Pi_h\Pi_\tau\Pi_h$.  Section~12 gives an independent Bargmann--Fock calculation of the leading overlap coefficient.

\subsection{Notation}\label{subsec:notation}
We use the following notation throughout.
\begin{itemize}[leftmargin=2.2em]
\item $Y=S(\mathcal O(-2))\simeq S^*\Sph^2\simeq\SO(3)$ is the line-bundle circle bundle.
\item $\partial M_\tau\simeq S^*\Sph^2\simeq S^1\Sph^2\simeq\SO(3)$ is the Grauert-tube boundary at radius $\tau>0$.
\item $\Pi_h$ and $\Pi_\tau$ are the orthogonal Szeg\H{o} projectors onto $H^2(Y)$ and $H^2(\partial M_\tau)$.
\item $V_\ell\simeq\mathrm{Sym}^{2\ell}\C^2$ is the irreducible $\SO(3)$ representation of dimension $2\ell+1$.
\item $c_\ell(\tau)$ is the overlap and $s_\ell(\tau)=c_\ell(\tau)^2$ is the eigenvalue of $\mathsf T_\tau=\Pi_h\Pi_\tau\Pi_h$ on the $\ell$th Hardy summand.
\item $r(\tau)=(1-e^{-2\tau})/2$ is the base of the exponential factor $r(\tau)^\ell$ in \eqref{eq:intro-factorization}.
\item $C_0(\tau)=\sqrt{2\sinh(2\tau)}e^{-\tau}=C_{\BF}(\tau)^2$.
\end{itemize}
The inner product is linear in the first variable and conjugate-linear in the second.  For $v\ne0$, we write $\widehat v=v/\norm v$.  All normalized overlaps are computed in the multiplicity factor $V_\ell^*$, where the Peter--Weyl normalization factor $(2\ell+1)^{-1}$ cancels.

\section{Two Hardy quantizations of \texorpdfstring{$\Sph^2$}{S2}}

\subsection{The line-bundle Hardy space for \texorpdfstring{$\mathcal O(2)$}{O(2)}}

Let $M=\Sph^2\simeq\CP^1$ with its round metric.  Its canonical line bundle is $K_M\simeq\mathcal O(-2)$.  Put $L=K_M^{-1}\simeq\mathcal O(2)$ and equip $L$ with its standard positive Hermitian metric.  Let
\[
  Y=S(K_M)\subset K_M\simeq L^*
\]
be the unit circle bundle.  By \eqref{eq:cosphere-canonical-bundle}, $Y$ is naturally identified with $S^*\Sph^2$.  It carries the standard Boothby--Wang contact form and the strictly pseudoconvex CR structure induced from the unit disk bundle in $L^*$.  The fiberwise tensor-square map on $\mathcal O(-1)$ restricts on unit circles to the degree-two map $e^{\ii\theta}\mapsto e^{2\ii\theta}$, whose kernel is $\{\pm1\}$.  Hence
\[
  Y\simeq \Sph^3/\{\pm1\}.
\]

The double cover $\SU(2)\to\SO(3)$ identifies $\Sph^3/\{\pm 1\}$ with $\SO(3)$.  We fix the resulting identification
\[
  Y\simeq S^*\Sph^2\simeq\SO(3).
\]
The CR structure on $Y$ is the horizontal lift of the complex structure on $\CP^1$.

The Hardy space $H^2(Y)$ is the closed subspace of $L^2(Y)$ consisting of boundary values of holomorphic functions on the unit disk bundle of $L^*$.  Its nonnegative circle weights are naturally identified with $H^0(\CP^1,L^\ell)=H^0(\CP^1,\mathcal O(2\ell))$.  Under the $\SO(3)$ identification, this is a left-invariant CR Hardy space.

\medskip

\subsection{The Grauert-tube Hardy space}

Let $M=\Sph^2$ with the round metric. The complexification of $M$ may be identified with the affine quadric
\[
Q^2=\{z\in\C^3\mid z\cdot z=1\}.
\]
The adapted complex structure on $TM$ identifies a Grauert-tube neighborhood of the zero section with a strictly pseudoconvex domain in $Q^2$. For each $\tau>0$ the boundary $\partial M_\tau$ is a smooth compact CR manifold equipped with the Szeg\H{o} projector $\Pi_\tau$ onto its Hardy space $H^2(\partial M_\tau)$.

We write $S^1M$ for the unit tangent bundle of a Riemannian manifold $M$.  The round metric identifies $S^*\Sph^2$ with $S^1\Sph^2$, and $S^1\Sph^2\simeq\SO(3)$.  For each $\tau>0$, the map
\[
  (x,\xi)\longmapsto \exp_x^{\C}\!\bigl(\ii\tau\xi^\sharp\bigr)
\]
identifies $S^*\Sph^2$ with $\partial M_\tau$.  We transport the Grauert CR structure to $\SO(3)$ through this map and view $H^2(\partial M_\tau)$ as a left-invariant Hardy space inside $L^2(\SO(3))$.

\begin{remark}
The contact form induced by the adapted complex structure may be written as $\alpha_\tau=\dd^c\rho$ for a defining function $\rho$ of the tube boundary.  The Reeb flow of $\alpha_\tau$ is the lifted geodesic flow on $S^1\Sph^2$.
\end{remark}

\section{Left-invariant CR structures and multiplicity freeness}

\subsection{Peter--Weyl decomposition}

Let $V_\ell$ be the irreducible $\SO(3)$ representation of dimension $2\ell+1$. The Peter--Weyl theorem gives an orthogonal decomposition
\begin{equation}
\label{eq:PW}
L^2(\SO(3))\simeq \bigoplus_{\ell\ge 0} V_\ell\otimes V_\ell^*.
\end{equation}
An $\SO(3)$ equivariant operator on $L^2(\SO(3))$ is block diagonal for \eqref{eq:PW} and acts on $V_\ell\otimes V_\ell^*$ as $\Id_{V_\ell}\otimes A_\ell$ for some linear map $A_\ell$ on $V_\ell^*$.

\begin{remark}[Normalization]
\label{rem:normalization}
Matrix coefficients satisfy the orthogonality relation
\[
  \int_{\SO(3)}D^\ell_{mn}(g)\,\overline{D^{\ell'}_{m'n'}(g)}\,\dd g
  =\frac{\delta_{\ell\ell'}\delta_{mm'}\delta_{nn'}}{2\ell+1}.
\]
The factor $(2\ell+1)^{-1}$ relates the $L^2(\SO(3))$ inner product to the abstract inner product on $V_\ell\otimes V_\ell^*$. Since all overlap and norm computations below take place inside a single $V_\ell^*$, this factor cancels in normalized quantities such as \eqref{eq:overlap-formula}.
\end{remark}

\medskip

\subsection{Left-invariant CR structures}

Let $G=\SO(3)$ and let $\mathfrak g=\mathfrak{so}(3)$. We fix the standard basis of $\mathfrak g$ given by the matrices
\begin{equation}
\label{eq:so3-matrices}
E_1=\begin{pmatrix}
0&0&0\\
0&0&-1\\
0&1&0
\end{pmatrix},\qquad
E_2=\begin{pmatrix}
0&0&1\\
0&0&0\\
-1&0&0
\end{pmatrix},\qquad
E_3=\begin{pmatrix}
0&-1&0\\
1&0&0\\
0&0&0
\end{pmatrix}.
\end{equation}
They satisfy
\begin{equation}
\label{eq:so3-bracket}
[E_1,E_2]=E_3,\qquad [E_2,E_3]=E_1,\qquad [E_3,E_1]=E_2.
\end{equation}
We write $L_X$ and $R_X$ for the left and right-invariant vector fields generated by $X\in\mathfrak g$.

A left-invariant CR structure on $G$ is determined by a complex line $\mathfrak q\subset\mathfrak g_\C$ satisfying $\mathfrak q\cap\overline{\mathfrak q}=\{0\}$. The corresponding $(0,1)$-operator on functions is the right-invariant operator $\overline L=R_{Z}$ where $Z$ spans $\mathfrak q$.

\begin{lemma}[Multiplicity one]
\label{lem:mult1}
Let $\mathfrak q\subset \mathfrak g_\C$ define a left-invariant strictly pseudoconvex
CR structure on $G=\SO(3)$. Then the corresponding Hardy space $\mathcal H\subset
L^2(G)$ is multiplicity-free in the Peter--Weyl decomposition. More precisely, for
each $\ell\in \N_0$, if
\[
\mathcal H_\ell:=\mathcal H\cap \bigl(V_\ell\otimes V_\ell^*\bigr),
\]
then either $\mathcal H_\ell=\{0\}$ or
\[
\mathcal H_\ell=V_\ell\otimes \C v_\ell
\]
for a uniquely determined line $\C v_\ell\subset V_\ell^*$.
\end{lemma}

\begin{proof}
Since $\dim G=3$, the CR dimension is one, so $\mathfrak q=\C Z$ for some nonzero $Z\in\mathfrak g_\C$.  The smooth vectors in the corresponding Hardy space are precisely the smooth CR functions and therefore satisfy
\[
  R_Zf=0.
\]
Every vector in a Peter--Weyl block is smooth.  Since right differentiation commutes with the left regular action, the intersection $\mathcal H_\ell$ is a left $G$-submodule and is the kernel of $R_Z$ on that block.

We now identify the block with $V_\ell\otimes V_\ell^*$ in the standard way. Under
this identification, the left regular action is on the first factor, while the right
regular action is on the second factor through the contragredient representation.
Differentiating the right action gives
\[
R_Z\big|_{V_\ell\otimes V_\ell^*}
=
\Id_{V_\ell}\otimes d\pi_\ell^*(Z).
\]
Hence
\[
\mathcal H_\ell
=
\ker\!\left(\Id_{V_\ell}\otimes d\pi_\ell^*(Z)\right)
=
V_\ell\otimes \ker d\pi_\ell^*(Z).
\]
Therefore the lemma reduces to showing that
\[
\dim \ker d\pi_\ell^*(Z)\le 1
\]
for every nonzero $Z\in \mathfrak g_\C$ and every irreducible representation
$\pi_\ell$. The differential of the contragredient representation is
\[
  d\pi_\ell^*(Z)=-d\pi_\ell(Z)^T.
\]
Transpose and multiplication by $-1$ preserve rank, so $d\pi_\ell^*(Z)$ and $d\pi_\ell(Z)$ have the same nullity.  It is therefore enough to prove
\[
  \dim\ker d\pi_\ell(Z)\le1.
\]

Now $\mathfrak g_\C\simeq \mathfrak{sl}_2(\C)$. Every nonzero element of
$\mathfrak{sl}_2(\C)$ is $\mathrm{Ad}(\mathrm{SL}_2(\C))$-conjugate either to a nonzero nilpotent
element or to a nonzero semisimple element. Accordingly there are two cases.

Assume first that $Z$ is nilpotent. Then $Z$ is $\mathrm{Ad}(\mathrm{SL}_2(\C))$-conjugate to the
standard raising operator
\[
E=
\begin{pmatrix}
0&1\\
0&0
\end{pmatrix}.
\]
Conjugating $d\pi_\ell(Z)$ by the induced intertwiner does not change the kernel
dimension, so it is enough to compute for $d\pi_\ell(E)$. In the irreducible
$\mathfrak{sl}_2$-module $V_\ell$ of dimension $2\ell+1$, one may choose a weight
basis
\[
v_{-\ell},v_{-\ell+1},\dots,v_{\ell}
\]
such that
\[
d\pi_\ell(E)v_m = c_m\,v_{m+1},
\qquad
c_m\neq 0 \ \text{for}\ m<\ell,
\qquad
d\pi_\ell(E)v_\ell=0.
\]
Thus
\[
\ker d\pi_\ell(E)=\C v_\ell
\]
is one-dimensional. Hence $\dim\ker d\pi_\ell(Z)=1$ when $Z$ is nilpotent.

Assume next that $Z$ is semisimple. Then $Z$ is $\mathrm{Ad}(\mathrm{SL}_2(\C))$-conjugate to
$aH$ for some $a\in \C^\times$, where
\[
H=
\begin{pmatrix}
1&0\\
0&-1
\end{pmatrix}.
\]
Again, conjugation does not change kernel dimension, so it is enough to compute
for $a\,d\pi_\ell(H)$. In the same weight basis,
\[
d\pi_\ell(H)v_m = 2m\,v_m,
\qquad
m=-\ell,-\ell+1,\dots,\ell.
\]
Each weight space is one-dimensional.  Since $\ell\in\N_0$ for an $\SO(3)$ representation, the weight $0$ occurs exactly once.  Consequently
\[
\dim\ker d\pi_\ell(Z)\le 1
\]
also in the semisimple case.

We have proved that $\ker d\pi_\ell^*(Z)$ is either $\{0\}$ or a complex line.
Choosing any nonzero vector $v_\ell\in \ker d\pi_\ell^*(Z)$ when this kernel is
nontrivial, we obtain
\[
\mathcal H_\ell
=
V_\ell\otimes \ker d\pi_\ell^*(Z)
=
V_\ell\otimes \C v_\ell.
\]
This proves the claim. The case $\ell=0$ is included automatically, since
$V_0\simeq \C$ and $d\pi_0(Z)=0$, so $\mathcal H_0=V_0\otimes V_0^*$ is
one-dimensional.
\end{proof}

\section{Explicit CR operators for the Hopf and Grauert structures}

\subsection{The Hopf CR operator}

The CR-structure coming from $\mathcal O(2)$ is the horizontal lift of the complex structure on $\CP^1\simeq \SO(3)/\SO(2)$. Under our identifications, the corresponding $(0,1)$ operator is the right-invariant operator
\begin{equation}
\label{eq:Lbar-hopf}
\overline L_h=R_{E_1+\ii E_2}.
\end{equation}
At the identity, the contact distribution is $\Span\{E_1,E_2\}$ and the induced complex structure $J_h$ satisfies $J_h(E_1)=E_2$ and $J_h(E_2)=-E_1$. It follows that $E_1+\ii E_2$ spans the $(-\ii)$ eigenspace, hence determines the $(0,1)$ operator.
Proposition \ref{prop:hopf-kernel} gives an equivalent verification by identifying $\ker\dd\pi_\ell(E_1+\ii E_2)$ with the highest weight line.

\medskip

\subsection{The Grauert-tube CR operator}

The Grauert-tube boundary of the round sphere is homogeneous, hence its transported CR structure on $\SO(3)$ is left-invariant. We compute the corresponding complex line in $\mathfrak g_\C$.

\begin{proposition}[Grauert $(0,1)$-generator]
\label{prop:Lbar-grauert}
Let $\tau>0$ and set $\alpha=\Coth\tau$.  Under the identification
\[
  \SO(3)\simeq S^1\Sph^2,
  \qquad
  g\longmapsto (ge_3,ge_1),
\]
the transported Grauert-CR structure has contact plane
\[
  H_\tau|_e=\Span\{E_1,E_3\}
\]
and complex structure
\[
  J_\tau(E_1)=-\alpha E_3,
  \qquad
  J_\tau(E_3)=\alpha^{-1}E_1.
\]
Consequently its $(0,1)$ line at the identity is spanned by
\[
  E_1-\ii\alpha E_3,
\]
and the corresponding left-invariant CR operator on functions is
\begin{equation}
\label{eq:Lbar-grauert}
  \overline L_\tau=R_{E_1-\ii\alpha E_3}.
\end{equation}
\end{proposition}

\begin{proof}
Use the quadric model
\[
  Q^2=\{z\in\C^3:z\cdot z=1\}
\]
and the radius-$\tau$ boundary point
\[
  z_\tau=\ii\sinh\tau\,e_1+\cosh\tau\,e_3.
\]
The radius-$\tau$ complexification map, written on $S^1\Sph^2$, is
\[
  \Phi_\tau(g)=gz_\tau
  =\ii\sinh\tau\,ge_1+\cosh\tau\,ge_3.
\]
At the identity,
\[
  d\Phi_\tau(X)=Xz_\tau
  =\ii\sinh\tau\,Xe_1+\cosh\tau\,Xe_3.
\]
The matrices in \eqref{eq:so3-matrices} give
\[
E_1e_3=-e_2,
\quad E_1e_1=0,
\quad E_2e_3=e_1,
\quad E_2e_1=-e_3,
\quad E_3e_3=0,
\quad E_3e_1=e_2.
\]
Hence
\[
  d\Phi_\tau(E_1)=-\cosh\tau\,e_2,
  \qquad
  d\Phi_\tau(E_3)=\ii\sinh\tau\,e_2,
  \qquad
  d\Phi_\tau(E_2)=\cosh\tau\,e_1-\ii\sinh\tau\,e_3.
\]
The contact form on $S^1\Sph^2$ is the restriction of the canonical one-form, up to a positive scalar depending only on the radius.  At $(e_3,e_1)$ it is proportional to
\[
  (\delta x,\delta v)\longmapsto \langle e_1,\delta x\rangle.
\]
For the infinitesimal motion induced by $X\in\mathfrak{so}(3)$, this value is
\[
  \langle e_1,Xe_3\rangle.
\]
Therefore it vanishes on $E_1$ and $E_3$, and takes the value $1$ on $E_2$.  The contact plane is $\Span\{E_1,E_3\}$ and the Reeb direction is $E_2$.

Multiplication by $\ii$ in the quadric tangent space gives the CR complex structure.  Since
\[
  \ii\,d\Phi_\tau(E_1)
  =-\ii\cosh\tau\,e_2
  =-\Coth\tau\,d\Phi_\tau(E_3),
\]
and
\[
  \ii\,d\Phi_\tau(E_3)
  =-\sinh\tau\,e_2
  =\tanh\tau\,d\Phi_\tau(E_1),
\]
transporting back to $\Span\{E_1,E_3\}$ gives
\[
  J_\tau(E_1)=-\Coth\tau\,E_3,
  \qquad
  J_\tau(E_3)=\tanh\tau\,E_1.
\]
The $(-\ii)$-eigenspace of this complex structure is spanned by
\[
  E_1+\ii J_\tau(E_1)=E_1-\ii\Coth\tau\,E_3.
\]
This is the claimed $(0,1)$ generator.
\end{proof}

\begin{remark}[Two Reeb directions]
The Hopf and Grauert structures use different contact forms under the common identification with $\SO(3)$.  The Hopf Reeb direction is $E_3$ and its contact plane is $\Span\{E_1,E_2\}$.  The Grauert Reeb direction is $E_2$ and its contact plane is $\Span\{E_1,E_3\}$.  Section~\ref{sec:BF-leading} compares their linearized horizontal structures after identifying each with the standard Heisenberg plane.
\end{remark}

\section{Borel--Weil models and CR-kernel lines}

\subsection{The Borel-Weil realization of \texorpdfstring{$V_\ell$}{V ell}}

Let $n=2\ell$.  The irreducible $\SO(3)$ module $V_\ell$ is of real type and hence is linearly self-dual.  We fix an $\SO(3)$-equivariant unitary identification $V_\ell^*\simeq V_\ell$ and realize this common representation as the space of homogeneous polynomials of degree $n$ in $(z_0,z_1)\in\C^2$.
\[
  V_\ell^*\simeq V_\ell\simeq\mathrm{Sym}^{n}\C^2.
\]
Under this identification, the contragredient differential on the multiplicity factor is represented by the operators below.
In the affine chart $w=z_1/z_0$, a homogeneous polynomial $P(z_0,z_1)=z_0^n p(w)$ corresponds to a polynomial $p(w)$ of degree at most $n$.

We use the standard basis $E_+,E_-,H$ of $\mathfrak{sl}(2,\C)$ with commutators
\[
[H,E_+]=2E_+,\qquad [H,E_-]=-2E_-,\qquad [E_+,E_-]=H.
\]
In the affine model, these act by differential operators
\begin{equation}
\label{eq:sl2-affine}
  E_+=\partial_w,\qquad
  E_-=-w^2\partial_w+n w,\qquad
  H=-2w\partial_w+n.
\end{equation}

\medskip

\subsection{The identification \texorpdfstring{$\mathfrak{so}(3)_\C\simeq\mathfrak{sl}(2,\C)$}{so(3)C is sl(2)C}}

The complexification $\mathfrak g_\C$ is isomorphic to $\mathfrak{sl}(2,\C)$. We record an explicit choice compatible with \eqref{eq:so3-bracket}.

\begin{lemma}
\label{lem:so3-sl2}
Define elements of $\mathfrak g_\C$ by
\[
X_+=E_1+\ii E_2,\qquad X_-=-E_1+\ii E_2,\qquad X_0=2\ii E_3.
\]
Then $X_+,X_-,X_0$ satisfy the $\mathfrak{sl}(2,\C)$ commutation relations
\[
[X_0,X_+]=2X_+,\qquad [X_0,X_-]=-2X_-,\qquad [X_+,X_-]=X_0.
\]
\end{lemma}

\begin{proof}
Using \eqref{eq:so3-bracket},
\[
 [X_0,X_+]
 =2\ii[E_3,E_1+\ii E_2]
 =2\ii(E_2-\ii E_1)=2(E_1+\ii E_2)=2X_+.
\]
Similarly,
\[
 [X_0,X_-]
 =2\ii[E_3,-E_1+\ii E_2]
 =2\ii(-E_2-\ii E_1)=-2(-E_1+\ii E_2)=-2X_-.
\]
Finally,
\[
 [X_+,X_-]
 =[E_1+\ii E_2,-E_1+\ii E_2]
 =\ii[E_1,E_2]-\ii[E_2,E_1]=2\ii E_3=X_0.
\]
\end{proof}

Lemma \ref{lem:so3-sl2} determines a Lie algebra isomorphism $\mathfrak g_\C\to\mathfrak{sl}(2,\C)$ by sending $X_\pm\mapsto E_\pm$ and $X_0\mapsto H$. It follows that in the representation $V_\ell$ we may write
\begin{equation}
\label{eq:so3-to-sl2}
\dd\pi_\ell(E_1+\ii E_2)=\dd\pi_\ell(X_+)=\dd\pi_\ell(E_+),\qquad
\dd\pi_\ell(E_3)=\frac{1}{2\ii}\,\dd\pi_\ell(H).
\end{equation}

We endow $\mathrm{Sym}^{n}\C^2$ with the unique $\U(2)$ invariant Hermitian inner product normalized by $\norm{z_0^n}=1$. With respect to the monomial basis,
\begin{equation}
\label{eq:norm-monomials}
  \norm{z_0^{n-k}z_1^k}^2 = \binom{n}{k}^{-1}.
\end{equation}

The beta-function computation on $\Sph^3\subset\C^2$ gives
\[
\int_{\Sph^3}|z_0|^{2(n-k)}|z_1|^{2k}\,d\sigma
=\frac{(n-k)!\,k!}{(n+1)!}\cdot\mathrm{vol}(\Sph^3).
\]
With the normalization $\norm{z_0^n}=1$ one obtains
$\norm{z_0^{n-k}z_1^k}^2/\norm{z_0^n}^2=\binom{n}{k}^{-1}$.

\medskip

\subsection{Kernel lines for the Hopf and Grauert CR operators}

\begin{proposition}[Hopf-kernel line]
\label{prop:hopf-kernel}
Let $\overline L_h$ be as in \eqref{eq:Lbar-hopf}. Then in $V_\ell\simeq\mathrm{Sym}^{2\ell}\C^2$ one has
\[
  \ker\dd\pi_\ell(E_1+\ii E_2)=\C\,z_0^{2\ell}.
\]
In particular one may take $u_\ell=z_0^{2\ell}$.
\end{proposition}

\begin{proof}
By \eqref{eq:so3-to-sl2} and \eqref{eq:sl2-affine}, $\dd\pi_\ell(E_1+\ii E_2)$ acts as $\partial_w$ on the affine polynomial $p(w)$. The kernel consists of constants, hence the corresponding homogeneous polynomial is $z_0^{2\ell}$.
\end{proof}

\begin{proposition}[Grauert-kernel line]
\label{prop:grauert-kernel}
Let $\overline L_\tau$ be as in \eqref{eq:Lbar-grauert} and set $\alpha=\Coth\tau$. Then in $V_\ell\simeq\mathrm{Sym}^{2\ell}\C^2$ one has
\begin{equation}
\label{eq:w-ell}
  \ker\dd\pi_\ell(E_1-\ii\alpha E_3)=\C\,(z_0^2+z_1^2+2\alpha z_0z_1)^\ell.
\end{equation}
In particular one may take
\[
  w_\ell(\tau)=(z_0^2+z_1^2+2\Coth\tau\,z_0z_1)^\ell.
\]
\end{proposition}

\begin{proof}
Let $n=2\ell$.  Using \eqref{eq:so3-to-sl2},
\[
\dd\pi_\ell(E_1-\ii\alpha E_3)
=\frac12(E_+-E_-)-\frac{\alpha}{2}H.
\]
In the affine model \eqref{eq:sl2-affine}, this operator sends $p(w)$ to
\[
\left(\frac{1+w^2}{2}+\alpha w\right)p'(w)-\frac n2(w+\alpha)p(w).
\]
Thus $p$ lies in the kernel precisely when
\begin{equation}
\label{eq:ode-grauert}
  (1+w^2+2\alpha w)p'(w)-n(w+\alpha)p(w)=0.
\end{equation}
Let
\[
  D_\alpha(w)=1+w^2+2\alpha w.
\]
Since $D_\alpha'(w)=2(w+\alpha)$ and $n=2\ell$, equation \eqref{eq:ode-grauert} is
\[
  D_\alpha(w)p'(w)-\ell D_\alpha'(w)p(w)=0.
\]
On the open set where $D_\alpha\ne0$ this is
\[
  \left(\frac{p(w)}{D_\alpha(w)^\ell}\right)'=0.
\]
Hence $p(w)=C D_\alpha(w)^\ell$ on that open set, and therefore everywhere by polynomial identity.  Homogenizing gives
\[
  z_0^{2\ell}D_\alpha(z_1/z_0)^\ell
  =(z_0^2+z_1^2+2\alpha z_0z_1)^\ell.
\]
This proves \eqref{eq:w-ell}.  Uniqueness of the line follows from Lemma~\ref{lem:mult1}.
\end{proof}

\section{Norms, Legendre polynomials, and overlap coefficients}

The Legendre polynomial enters through the squared-binomial sum
\[
  \sum_{j=0}^{\ell}\frac{\binom{\ell}{j}^2}{\binom{2\ell}{2j}}t^j.
\]
The generating function of this sum reduces the norm calculation to a single value of $\Leg_\ell$.

\subsection{Takagi diagonalization of the quadratic form}

Let $\alpha>1$ and consider the complex symmetric matrix
\[
  Q_\alpha=\begin{pmatrix}1&\alpha\\ \alpha&1\end{pmatrix}
\]
so that
\[
  z^TQ_\alpha z=z_0^2+z_1^2+2\alpha z_0z_1.
\]
The Takagi singular values of $Q_\alpha$ are $\alpha+1$ and $\alpha-1$.

\begin{lemma}[Takagi factorization]
\label{lem:takagi}
For $\alpha>1$, the matrix
\[
  U=\frac1{\sqrt2}\begin{pmatrix}1&\ii\\ 1&-\ii\end{pmatrix}
\]
is unitary and satisfies
\begin{equation}
\label{eq:takagi}
  U^TQ_\alpha U=\diag(\alpha+1,\alpha-1).
\end{equation}
\end{lemma}

\begin{proof}
The columns of $U$ are orthonormal, and direct multiplication gives
\[
  \frac12(1,1)Q_\alpha\binom{1}{1}=\alpha+1,
  \qquad
  \frac12(\ii,-\ii)Q_\alpha\binom{\ii}{-\ii}=\alpha-1,
\]
and the mixed term vanishes.  This proves \eqref{eq:takagi}.
\end{proof}

\subsection{Exact norm formula}

\begin{proposition}[Norm of the Grauert kernel vector]
\label{prop:norm-formula}
Fix $\ell\ge0$ and set $\alpha=\Coth\tau>1$.
With respect to the normalization $\norm{z_0^{2\ell}}=1$, the vector
\[
 w_\ell(\tau)=(z_0^2+z_1^2+2\alpha z_0z_1)^\ell
\]
has squared norm
\begin{equation}
\label{eq:norm-legendre}
  \norm{w_\ell(\tau)}^2
  = \frac{4^\ell}{\binom{2\ell}{\ell}}\,\Csch^{2\ell}\!\tau\,\Leg_\ell(\cosh 2\tau).
\end{equation}
\end{proposition}

\begin{proof}
By Lemma~\ref{lem:takagi} and the $\U(2)$ invariance of the inner product, we may compute the norm after the unitary change of variables which sends the quadratic form to
\[
  \sigma_1z_0^2+\sigma_2z_1^2,
  \qquad
  \sigma_1=\alpha+1,
  \quad
  \sigma_2=\alpha-1.
\]
Expanding gives
\begin{equation}
\label{eq:norm-sum}
\norm{w_\ell(\tau)}^2
=\sum_{j=0}^{\ell}\binom{\ell}{j}^2\sigma_1^{2(\ell-j)}\sigma_2^{2j}\binom{2\ell}{2j}^{-1}.
\end{equation}
The identity
\begin{equation}
\label{eq:comb-identity}
  \frac{\binom{\ell}{j}^2}{\binom{2\ell}{2j}}
  =\frac{(\ell!)^2}{(2\ell)!}\binom{2j}{j}\binom{2(\ell-j)}{\ell-j}
\end{equation}
follows by expanding all binomial coefficients.  Put
\[
  t=\left(\frac{\sigma_2}{\sigma_1}\right)^2.
\]
Then \eqref{eq:norm-sum} becomes
\[
\norm{w_\ell(\tau)}^2
=\frac{(\ell!)^2}{(2\ell)!}\sigma_1^{2\ell}
\sum_{j=0}^{\ell}\binom{2j}{j}\binom{2(\ell-j)}{\ell-j}t^j.
\]
The coefficient of $x^\ell$ in
\[
  \frac1{\sqrt{(1-4x)(1-4tx)}}
\]
is the sum in the last display.  Comparing this generating function with
\[
  \frac1{\sqrt{1-2Xz+z^2}}=\sum_{\ell\ge0}\Leg_\ell(X)z^\ell
\]
gives
\[
\sum_{j=0}^{\ell}\binom{2j}{j}\binom{2(\ell-j)}{\ell-j}t^j
=(4\sqrt t)^\ell\Leg_\ell\left(\frac{1+t}{2\sqrt t}\right).
\]
Hence
\[
\norm{w_\ell(\tau)}^2
=\frac{4^\ell}{\binom{2\ell}{\ell}}(\sigma_1\sigma_2)^\ell
\Leg_\ell\left(\frac{\sigma_1^2+\sigma_2^2}{2\sigma_1\sigma_2}\right).
\]
For $\alpha=\Coth\tau$,
\[
  \sigma_1\sigma_2=\alpha^2-1=\Csch^2\tau,
  \qquad
  \frac{\sigma_1^2+\sigma_2^2}{2\sigma_1\sigma_2}
  =\frac{\alpha^2+1}{\alpha^2-1}=\cosh 2\tau.
\]
Substitution proves \eqref{eq:norm-legendre}.
\end{proof}

\subsection{Overlap and limiting regimes}

\begin{proof}[Proof of the overlap formula in Theorem~\ref{thm:main}]
The weight decomposition for the diagonal circle action on $\C^2$ makes the monomials $z_0^{2\ell-k}z_1^k$ pairwise orthogonal for the $\U(2)$ invariant inner product. Write
\[
w_\ell(\tau)=\sum_{k=0}^{2\ell} a_k\,z_0^{2\ell-k}z_1^k.
\]
Evaluating at $z_1=0$ gives $w_\ell(\tau)(z_0,0)=z_0^{2\ell}$, hence $a_0=1$. Since $u_\ell=z_0^{2\ell}$ has norm one and is orthogonal to all monomials with $k\ge1$, one has
\[
\ip{w_\ell(\tau)}{u_\ell}=a_0\ip{z_0^{2\ell}}{z_0^{2\ell}}=1.
\]
Therefore
\[
\big|\ip{\widehat w_\ell(\tau)}{\widehat u_\ell}\big|=\norm{w_\ell(\tau)}^{-1}.
\]
Taking the reciprocal square root of \eqref{eq:norm-legendre} yields \eqref{eq:overlap-formula}.
\end{proof}

\begin{remark}[Fubini--Study distance]
\label{rem:FS-distance}
The angle
\[
  \theta_\ell(\tau)
  =\arccos\bigl|\ip{\widehat w_\ell(\tau)}{\widehat u_\ell}\bigr|
\]
is the Fubini--Study distance between the two lines in $\mathbb P(V_\ell^*)$, and $s_\ell(\tau)=\cos^2\theta_\ell(\tau)$.  For fixed $\tau$, this squared overlap decays exponentially with $\ell$.  Section~\ref{sec:BF-leading} identifies the limit $r(\tau)^{-2\ell}s_\ell(\tau)\to C_{\BF}(\tau)^2$ with a Bargmann--Fock Gaussian matrix coefficient.
\end{remark}

\begin{proposition}[Limits and small-radius expansion]
\label{prop:tau-limits}
Fix $\ell\ge1$.
As $\tau\downarrow 0$,
\[
\big|\ip{\widehat w_\ell(\tau)}{\widehat u_\ell}\big|
=\frac{\sqrt{\binom{2\ell}{\ell}}}{2^\ell}\,\tau^\ell\Big(1-\frac{\ell(3\ell+2)}{6}\,\tau^2+O(\tau^4)\Big).
\]
As $\tau\to\infty$,
\[
\big|\ip{\widehat w_\ell(\tau)}{\widehat u_\ell}\big|=2^{-\ell}\big(1+O(e^{-2\tau})\big).
\]
\end{proposition}

\begin{proof}
Start from \eqref{eq:overlap-formula}.
As $\tau\downarrow 0$ one has
\[
\sinh\tau=\tau+\frac{\tau^3}{6}+O(\tau^5),
\qquad
\cosh 2\tau=1+2\tau^2+\frac{2}{3}\tau^4+O(\tau^6).
\]
Set $x=\cosh 2\tau$.
Since $\Leg_\ell$ is analytic at $x=1$ and $\Leg_\ell(1)=1$,
\[
\Leg_\ell(x)=1+\Leg_\ell'(1)(x-1)+O\big((x-1)^2\big).
\]
The derivative at $1$ is
\begin{equation}
\label{eq:Pprime1}
\Leg_\ell'(1)=\frac{\ell(\ell+1)}{2}.
\end{equation}
To verify \eqref{eq:Pprime1}, differentiate the three-term recurrence
\[
(\ell+1)\Leg_{\ell+1}(x)=(2\ell+1)x\Leg_\ell(x)-\ell\Leg_{\ell-1}(x)
\]
and evaluate at $x=1$ using $\Leg_k(1)=1$.
Writing $a_\ell=\Leg_\ell'(1)$ gives
\[
(\ell+1)a_{\ell+1}=(2\ell+1)(1+a_\ell)-\ell a_{\ell-1}.
\]
With $a_0=0$ and $a_1=1$, the identity $a_\ell=\ell(\ell+1)/2$ follows by induction.

Since $x-1=2\tau^2+\frac{2}{3}\tau^4+O(\tau^6)$, and
\[
\Leg_\ell''(1)=\frac{\ell(\ell-1)(\ell+1)(\ell+2)}{8},
\]
one obtains
\[
\Leg_\ell(\cosh 2\tau)=1+\ell(\ell+1)\tau^2+\frac{\ell(\ell+1)(3\ell^2+3\ell-2)}{12}\tau^4+O(\tau^6).
\]
Consequently
\[
\Leg_\ell(\cosh 2\tau)^{-1/2}=1-\frac{\ell(\ell+1)}{2}\tau^2+\frac{\ell(\ell+1)(3\ell^2+3\ell+1)}{12}\tau^4+O(\tau^6).
\]
Also
\[
\sinh^\ell\!\tau=\tau^\ell\left(1+\frac{\ell}{6}\tau^2+\frac{\ell(5\ell-2)}{360}\tau^4+O(\tau^6)\right).
\]
Substituting into \eqref{eq:overlap-formula} gives
\[
\frac{\sinh^\ell\!\tau}{\sqrt{\Leg_\ell(\cosh 2\tau)}}
=\tau^\ell\left(1-\frac{\ell(3\ell+2)}{6}\tau^2+\frac{\ell(90\ell^3+150\ell^2+95\ell+28)}{360}\tau^4+O(\tau^6)\right).
\]
Multiplying by $\sqrt{\binom{2\ell}{\ell}}/2^\ell$ yields the stated small $\tau$ expansion.

As $\tau\to\infty$ one has
\[
\sinh\tau=\frac12 e^{\tau}\big(1+O(e^{-2\tau})\big),
\qquad
\cosh 2\tau=\frac12 e^{2\tau}\big(1+O(e^{-4\tau})\big).
\]
For fixed $\ell$, the leading term of the Legendre polynomial is
\begin{equation}
\label{eq:leg-leading}
\Leg_\ell(x)=\frac{\binom{2\ell}{\ell}}{2^\ell}\,x^\ell+O(x^{\ell-1})
\qquad \text{as } x\to\infty.
\end{equation}
Substituting $x=\cosh 2\tau$ into \eqref{eq:leg-leading} gives
\[
\Leg_\ell(\cosh 2\tau)=\frac{\binom{2\ell}{\ell}}{2^\ell}\,\cosh^\ell(2\tau)\big(1+O(e^{-2\tau})\big).
\]
Taking square roots and substituting into \eqref{eq:overlap-formula} yields
\[
\big|\ip{\widehat w_\ell(\tau)}{\widehat u_\ell}\big|
=2^{-\ell/2}\,\frac{\sinh^\ell\!\tau}{\cosh^{\ell/2}(2\tau)}\,\big(1+O(e^{-2\tau})\big).
\]
Using $\cosh(2\tau)=2\sinh^2\!\tau+1$ gives
\[
\frac{\sinh^2\!\tau}{\cosh(2\tau)}=\frac12\Big(1+O(e^{-2\tau})\Big).
\]
Therefore
\[
2^{-\ell/2}\,\frac{\sinh^\ell\!\tau}{\cosh^{\ell/2}(2\tau)}
=2^{-\ell/2}\,\Big(\frac{\sinh^2\!\tau}{\cosh(2\tau)}\Big)^{\ell/2}
=2^{-\ell}\big(1+O(e^{-2\tau})\big)
\]
and the claimed limit follows.
\end{proof}

\begin{remark}[Limit as $\tau\to0$]
As $\tau\downarrow0$, the coefficient $\alpha=\Coth\tau$ diverges, and
\[
  (2\alpha)^{-\ell}\bigl(z_0^2+z_1^2+2\alpha z_0z_1\bigr)^\ell
  \longrightarrow (z_0z_1)^\ell.
\]
Thus the Grauert line converges projectively to the weight-zero line, which is orthogonal to the highest-weight line $\C z_0^{2\ell}$.
\end{remark}

\subsection{Large-\texorpdfstring{$\ell$}{ell} asymptotics}

\begin{corollary}[Large $\ell$ asymptotics]
\label{cor:asymptotics}
Fix $\tau>0$.
As $\ell\to\infty$,
\[
\big|\ip{\widehat w_\ell(\tau)}{\widehat u_\ell}\big|
= (2\sinh 2\tau)^{1/4}\,e^{-\tau/2}\,\Big(\frac{1-e^{-2\tau}}{2}\Big)^\ell\big(1+o(1)\big).
\]
\end{corollary}

\begin{proof}
Write \eqref{eq:overlap-formula} in the form
\begin{equation}
\label{eq:overlap-squared}
\big|\ip{\widehat w_\ell(\tau)}{\widehat u_\ell}\big|^2
=\frac{\binom{2\ell}{\ell}}{4^\ell}\,\frac{\sinh^{2\ell}\!\tau}{\Leg_\ell(\cosh 2\tau)}.
\end{equation}
Stirling's formula gives
\begin{equation}
\label{eq:central-binomial}
\frac{\binom{2\ell}{\ell}}{4^\ell}=\frac{1}{\sqrt{\pi\ell}}\big(1+o(1)\big).
\end{equation}
Set $\eta=2\tau$ and $x=\cosh\eta>1$.
The classical Szeg\H{o} asymptotic for Legendre polynomials off the cut gives
\begin{equation}
\label{eq:leg-szego}
\Leg_\ell(\cosh\eta)=\frac{e^{(\ell+1/2)\eta}}{\sqrt{2\pi\ell\sinh\eta}}\big(1+o(1)\big)
\qquad \text{as } \ell\to\infty.
\end{equation}
Substituting \eqref{eq:central-binomial} and \eqref{eq:leg-szego} into \eqref{eq:overlap-squared} yields
\begin{align*}
\big|\ip{\widehat w_\ell(\tau)}{\widehat u_\ell}\big|^2
&= \frac{1}{\sqrt{\pi\ell}}\,\sinh^{2\ell}\!\tau\,\sqrt{2\pi\ell\sinh 2\tau}\,e^{-(2\ell+1)\tau}\big(1+o(1)\big) \\
&= \sqrt{2\sinh 2\tau}\,e^{-\tau}\,\Big(\frac{\sinh\tau}{e^{\tau}}\Big)^{2\ell}\big(1+o(1)\big) \\
&= \sqrt{2\sinh 2\tau}\,e^{-\tau}\,\Big(\frac{1-e^{-2\tau}}{2}\Big)^{2\ell}\big(1+o(1)\big).
\end{align*}
Taking square roots yields the stated formula.
\end{proof}

\begin{remark}[Exponential decay rate]
\label{rem:complexification-ratio}
The exponential decay rate in Corollary~\ref{cor:asymptotics} is the constant $r(\tau)$ of \eqref{eq:intro-factorization}.  It increases from $0$ to $1/2$ as $\tau$ runs from $0$ to $\infty$, and it is related to the squeeze parameter $b_\tau$ of Section~\ref{sec:BF-leading} by $r(\tau)=(1-b_\tau)/2$.
\end{remark}

\subsection{Monotonicity in \texorpdfstring{$\ell$}{ell}}

\begin{proposition}
\label{prop:monotonicity}
Fix $\tau>0$. The overlap $\big|\ip{\widehat w_\ell(\tau)}{\widehat u_\ell}\big|$ is strictly decreasing in $\ell$ for $\ell\ge1$.
\end{proposition}

\begin{proof}
Set $x=\cosh(2\tau)>1$ and $s=\sinh^2\!\tau=(x-1)/2$.
From \eqref{eq:overlap-formula},
\[
\big|\ip{\widehat w_\ell(\tau)}{\widehat u_\ell}\big|^2
=\binom{2\ell}{\ell}\,\frac{s^\ell}{4^\ell\Leg_\ell(x)}.
\]
Write this quantity as $R_\ell$.  Then
\[
\frac{R_{\ell+1}}{R_\ell}
=\frac{\binom{2\ell+2}{\ell+1}}{\binom{2\ell}{\ell}}\,\frac{s}{4}\,\frac{\Leg_\ell(x)}{\Leg_{\ell+1}(x)}
=\frac{2\ell+1}{\ell+1}\,\frac{s}{2}\,\frac{\Leg_\ell(x)}{\Leg_{\ell+1}(x)}.
\]

We claim that
\[
\Leg_{\ell+1}(x)>x\,\Leg_\ell(x)
\qquad\text{for every }\ell\ge1\text{ and }x>1.
\]
For $\ell=1$ this is immediate from
\[
\Leg_2(x)-x\Leg_1(x)=\frac{x^2-1}{2}>0.
\]
Assume inductively that $\Leg_\ell(x)>x\,\Leg_{\ell-1}(x)$ for some $\ell\ge1$.  Then
\[
\Leg_{\ell-1}(x)<\frac{1}{x}\Leg_\ell(x),
\]
and the three-term recurrence gives
\[
(\ell+1)\Leg_{\ell+1}(x)
=(2\ell+1)x\Leg_\ell(x)-\ell \Leg_{\ell-1}(x)
>
\left((2\ell+1)x-\frac{\ell}{x}\right)\Leg_\ell(x).
\]
Since
\[
\left((2\ell+1)x-\frac{\ell}{x}\right)-(\ell+1)x
=\ell\left(x-\frac1x\right)
=\frac{\ell(x^2-1)}{x}>0,
\]
we obtain $(\ell+1)\Leg_{\ell+1}(x)>(\ell+1)x\Leg_\ell(x)$, and therefore
\[
\Leg_{\ell+1}(x)>x\Leg_\ell(x).
\]
The claim follows by induction.

Consequently
\[
\frac{\Leg_\ell(x)}{\Leg_{\ell+1}(x)}<\frac{1}{x},
\]
so
\[
\frac{R_{\ell+1}}{R_\ell}
<
\frac{2\ell+1}{\ell+1}\,\frac{s}{2x}
<
\frac{x-1}{2x}
<1.
\]
Hence $R_{\ell+1}<R_\ell$ for every $\ell\ge1$, which proves the monotonicity of the overlap.
\end{proof}

\begin{proposition}[Asymptotic ratio]
\label{prop:sharp-ratio}
Let $R_\ell=s_\ell(\tau)=\big|\ip{\widehat w_\ell(\tau)}{\widehat u_\ell}\big|^2$.  Then, as $\ell\to\infty$,
\[
\frac{R_{\ell+1}}{R_\ell}
=
r(\tau)^2\left(1+\frac{c(\tau)}{\ell^2}+\mathcal O_\tau(\ell^{-3})\right),
\qquad
c(\tau)=\frac{\Coth(2\tau)-1}{8}.
\]
\end{proposition}
\begin{proof}
By Proposition~\ref{prop:sell-expansion},
$R_\ell=C_0r^{2\ell}(1+a_1/\ell+a_2/\ell^2+O(\ell^{-3}))$.  Hence
\[
\frac{R_{\ell+1}}{R_\ell}=r^2\frac{1+a_1/(\ell+1)+a_2/(\ell+1)^2+O(\ell^{-3})}{1+a_1/\ell+a_2/\ell^2+O(\ell^{-3})}.
\]
Using $(\ell+1)^{-1}=\ell^{-1}-\ell^{-2}+O(\ell^{-3})$ and $(\ell+1)^{-2}=\ell^{-2}+O(\ell^{-3})$, the numerator equals
$1+a_1/\ell+(a_2-a_1)/\ell^2+O(\ell^{-3})$.  Dividing by the denominator gives
\[
\frac{R_{\ell+1}}{R_\ell}=r^2\left(1-\frac{a_1}{\ell^2}+O(\ell^{-3})\right).
\]
Since $a_1(\tau)=(1-\Coth(2\tau))/8$, the coefficient is $c(\tau)=-a_1(\tau)=(\Coth(2\tau)-1)/8$.
\end{proof}

\begin{table}[ht]
\centering
\begin{tabular}{ccccc}
\toprule
 & $\tau=0.5$ & $\tau=1.0$ & $\tau=1.5$ & $\tau=2.0$\\
\midrule
$\ell=1$ & 0.2966 & 0.4284 & 0.4745 & 0.4908\\
$\ell=2$ & 0.0949 & 0.1857 & 0.2255 & 0.2409\\
$\ell=3$ & 0.0302 & 0.0804 & 0.1072 & 0.1182\\
$\ell=5$ & 0.00303 & 0.0150 & 0.0242 & 0.0285\\
$\ell=10$ & $9.57\times10^{-6}$ & $2.27\times10^{-4}$ & $5.86\times10^{-4}$ & $8.12\times10^{-4}$\\
\bottomrule
\end{tabular}
\caption{Values of $\big|\langle\widehat{w}_\ell(\tau),\widehat{u}_\ell\rangle\big|$ computed from \eqref{eq:overlap-formula}.}
\label{tab:overlap-values}
\end{table}
\begin{remark}[Numerical illustration]\label{rem:numerical}
Table~\ref{tab:overlap-values} illustrates the exponential decay in $\ell$ for fixed $\tau$ and the approach to $2^{-\ell}$ as $\tau\to\infty$ predicted by Corollary~\ref{cor:asymptotics} and Proposition~\ref{prop:tau-limits}.
\end{remark}

\section{The intertwining operator}

\subsection{Block formula}
Let $\Pi_h$ and $\Pi_\tau$ denote the orthogonal Szeg\H{o} projectors onto $H^2(Y)$ and $H^2(\partial M_\tau)$.  On the $\ell$th Peter--Weyl block,
\[
  \Pi_h=\Id_{V_\ell}\otimes\ket{\widehat u_\ell}\bra{\widehat u_\ell},
  \qquad
  \Pi_\tau=\Id_{V_\ell}\otimes\ket{\widehat w_\ell(\tau)}\bra{\widehat w_\ell(\tau)}.
\]
It follows that
\begin{equation}
\label{eq:cross-block}
  \Pi_h\Pi_\tau
  =\ip{\widehat u_\ell}{\widehat w_\ell(\tau)}
   \Id_{V_\ell}\otimes\ket{\widehat u_\ell}\bra{\widehat w_\ell(\tau)}.
\end{equation}

\begin{definition}[Intertwining operator]
\label{def:intertwiner}
Define $\mathcal T_\tau\colon H^2(\partial M_\tau)\to H^2(Y)$ by
\[
  \mathcal T_\tau\bigl(v\otimes\widehat w_\ell(\tau)\bigr)
  =v\otimes\widehat u_\ell.
\]
Its zero extension to $L^2(\SO(3))$ has block formula
\begin{equation}
\label{eq:T-block}
  \mathcal T_\tau\big|_{V_\ell\otimes V_\ell^*}
  =\Id_{V_\ell}\otimes\ket{\widehat u_\ell}\bra{\widehat w_\ell(\tau)}.
\end{equation}
\end{definition}

\begin{proof}[Proof of Theorem~\ref{thm:intertwiner}]
For each $\ell$, both Hardy summands are irreducible left $\SO(3)$ modules isomorphic to $V_\ell$.  Formula~\eqref{eq:T-block} is unitary from $V_\ell\otimes\C\widehat w_\ell(\tau)$ onto $V_\ell\otimes\C\widehat u_\ell$.  Their Hilbert direct sum is therefore a unitary from $H^2(\partial M_\tau)$ onto $H^2(Y)$.  The zero extension is a partial isometry with initial projection $\Pi_\tau$ and final projection $\Pi_h$.  The prescribed value on each normalized multiplicity vector fixes the scalar in Schur's lemma, which proves uniqueness.
\end{proof}

\subsection{Trace class of \texorpdfstring{$\Pi_h\Pi_\tau$}{the product of the Hardy projectors}}
\begin{proposition}
\label{prop:cross-trace-class}
For every $\tau>0$, the operator $\Pi_h\Pi_\tau$ is trace class.  Its singular value on the $\ell$th Hardy summand is $c_\ell(\tau)$ with multiplicity $2\ell+1$.
\end{proposition}

\begin{proof}
Formula~\eqref{eq:cross-block} gives the singular values.  Corollary~\ref{cor:asymptotics} yields $c_\ell(\tau)=O_\tau(r(\tau)^\ell)$ with $0<r(\tau)<1$, and hence
\[
  \sum_{\ell\ge0}(2\ell+1)c_\ell(\tau)<\infty.
\]
\end{proof}

\subsection{Dependence on the Grauert radius}
\begin{theorem}[Monotonicity in the radius]
\label{thm:radius-monotonicity}
For every fixed $\ell\ge1$, the function $\tau\mapsto s_\ell(\tau)$ is strictly increasing on $(0,\infty)$ and maps this interval onto $(0,2^{-2\ell})$.
\end{theorem}

\begin{proof}
Set $x=\cosh 2\tau$.  From \eqref{eq:overlap-squared},
\[
  \log s_\ell(\tau)
  =\log\left(\frac{\binom{2\ell}{\ell}}{4^\ell}\right)
   +2\ell\log\sinh\tau-\log\Leg_\ell(x).
\]
The identity
\[
  (x^2-1)\Leg_\ell'(x)=\ell\bigl(x\Leg_\ell(x)-\Leg_{\ell-1}(x)\bigr)
\]
gives
\begin{equation}
\label{eq:tau-monotone-derivative}
  \frac{d}{d\tau}\log s_\ell(\tau)
  =\frac{2\ell}{\sinh 2\tau}
    \left(1+\frac{\Leg_{\ell-1}(x)}{\Leg_\ell(x)}\right)>0,
\end{equation}
where positivity follows from $\Leg_j(x)>0$ for $x>1$.  Proposition~\ref{prop:tau-limits} gives the endpoint limits $0$ and $2^{-2\ell}$.
\end{proof}

Thus, for any fixed $\ell\ge1$, the single eigenvalue $s_\ell(\tau)$ determines $\tau$ within this one-parameter family, although the round Laplace spectrum itself does not vary.

\section{The operator \texorpdfstring{$\Pi_h\Pi_\tau\Pi_h$}{Pi h Pi tau Pi h}}
\begin{definition}
For the two Hardy projectors, set
\begin{equation}
\label{eq:T-operator}
  \mathsf T_\tau=\Pi_h\Pi_\tau\Pi_h\colon H^2(Y)\longrightarrow H^2(Y).
\end{equation}
\end{definition}

\begin{proposition}[Spectrum of $\Pi_h\Pi_\tau\Pi_h$]
\label{prop:T-spectrum}
The operator $\mathsf T_\tau$ is positive, trace class, and $\SO(3)$ equivariant.  On $V_\ell\otimes\C\widehat u_\ell$ it acts by the scalar
\[
  s_\ell(\tau)=\bigl|\ip{\widehat w_\ell(\tau)}{\widehat u_\ell}\bigr|^2,
\]
which has multiplicity $2\ell+1$.
\end{proposition}

\begin{proof}
For $v\in H^2(Y)$,
\[
  \ip{\mathsf T_\tau v}{v}=\norm{\Pi_\tau v}^2,
\]
so $0\le\mathsf T_\tau\le\Id$.  On the $\ell$th multiplicity line, the product of the three rank-one projectors is
\[
  \ket{\widehat u_\ell}\bra{\widehat u_\ell}
  \ket{\widehat w_\ell(\tau)}\bra{\widehat w_\ell(\tau)}
  \ket{\widehat u_\ell}\bra{\widehat u_\ell}
  =s_\ell(\tau)\ket{\widehat u_\ell}\bra{\widehat u_\ell}.
\]
The multiplicity is $\dim V_\ell=2\ell+1$.  Trace-class summability follows from the exponential bound on $s_\ell(\tau)$.
\end{proof}

\begin{remark}
Because both Hardy spaces are multiplicity-free, $\mathsf T_\tau$ is diagonal on the Hardy decomposition, with eigenvalues $s_\ell(\tau)$.
\end{remark}

\section{Trace formulas with the Hardy projectors}\label{sec:trace-formulas}

\subsection{Notation and phase convention}
Let
\[
  Y=\SO(3),\qquad g_t=\exp(-tE_2).
\]
We define the right-translation operator $V_g^t$ by its Peter--Weyl block action
\[
  V_g^t\big|_{V_\ell\otimes V_\ell^*}
  =\Id_{V_\ell}\otimes\pi_\ell^*(g_t),
  \qquad \ell\in\N_0,
\]
and set $\widetilde V_g^t=e^{\ii t/2}V_g^t$.  The scalar prefactor $e^{\ii t/2}$ multiplies each of the five traces below by the same factor, so it cancels in every comparison between them.  Its only visible effect is the sign $e^{\ii\pi k}=(-1)^k$ at a geodesic period $t=2\pi k$.  The flat trace below is understood as the distributional Abel limit in Theorem~\ref{thm:trace-dirichlet}.

Let \(\widehat u_\ell\in V_\ell^*\) be the normalized Hopf kernel vector and let \(\widehat w_\ell(\tau)\in V_\ell^*\) be the normalized Grauert kernel vector of Proposition~\ref{prop:grauert-kernel}.  We choose their phases so that \(\ip{\widehat u_\ell}{\widehat w_\ell(\tau)}=c_\ell(\tau)>0\), with $c_\ell(\tau)$ given by \eqref{eq:overlap-formula}.

\begin{theorem}[Exact spectral trace series]
\label{thm:trace-dirichlet}
For each fixed \(\tau>0\), the full flat trace and the four traces with Hardy projectors have the following Abel-regularized spectral expansions.

\emph{Full flat trace.}
\begin{equation}\label{eq:C-full-trace-Abel}
  \Tr^{\flat}(\widetilde V_g^t)
  =
  \lim_{\rho\uparrow1}
  e^{\ii t/2}\sum_{\ell=0}^{\infty}
  \rho^\ell(2\ell+1)\,\chi_\ell(g_t)
\end{equation}
where the limit is taken in \(\mathcal D'(\R)\), and
\begin{equation}\label{eq:C-character}
  \chi_\ell(g_t)
  =
  \frac{\sin\bigl((2\ell+1)t/2\bigr)}{\sin(t/2)}
\end{equation}
away from \(2\pi\mathbb Z\), with the usual continuous character value
\(\chi_\ell(e)=2\ell+1\) at periods.

\emph{Hopf diagonal trace.}
\begin{equation}\label{eq:C-Hopf-trace-Abel}
  T_{hh}(t)
  :=
  \Tr(\Pi_h\widetilde V_g^t\Pi_h)
  =
  \lim_{\rho\uparrow1}
  e^{\ii t/2}\sum_{\ell=0}^{\infty}
  \rho^\ell(2\ell+1)\cos^{2\ell}\!\frac t2.
\end{equation}
For \(t\notin2\pi\mathbb Z\), this Abel limit is the ordinary smooth function
\begin{equation}\label{eq:C-Hopf-closed}
  T_{hh}(t)
  =
  e^{\ii t/2}
  \frac{1+\cos^2(t/2)}{\bigl(1-
  \cos^2(t/2)\bigr)^2}.
\end{equation}

\emph{Grauert-diagonal trace.}
\begin{equation}\label{eq:C-Grauert-trace-Abel}
  T_{\tau\tau}(t)
  :=
  \Tr(\Pi_\tau\widetilde V_g^t\Pi_\tau)
  =
  \lim_{\rho\uparrow1}
  e^{\ii t/2}\sum_{\ell=0}^{\infty}
  \rho^\ell(2\ell+1)
  a_{\tau,\ell}(t),
\end{equation}
where
\begin{equation}\label{eq:C-a-tau-def}
  a_{\tau,\ell}(t)
  =
  \ip{\widehat w_\ell(\tau)}{\pi_\ell^*(g_t)\widehat w_\ell(\tau)}.
\end{equation}

\emph{Mixed traces.}
\begin{equation}\label{eq:C-cross-trace-real}
  T_{h\tau}(t)
  :=
  \Tr(\Pi_h\widetilde V_g^t\Pi_\tau)
  =
  e^{\ii t/2}\sum_{\ell=0}^{\infty}
  (2\ell+1)s_\ell(\tau)
  \bigl(1-\Coth\tau\,\sin t\bigr)^\ell,
\end{equation}
with absolute and uniform convergence for \(t\in\R\).  The opposite mixed trace is
\begin{equation}\label{eq:C-cross-opposite-real}
  T_{\tau h}(t)
  :=
  \Tr(\Pi_\tau\widetilde V_g^t\Pi_h)
  =
  e^{\ii t/2}\sum_{\ell=0}^{\infty}
  (2\ell+1)s_\ell(\tau)
  \bigl(1+\Coth\tau\,\sin t\bigr)^\ell,
\end{equation}
again with absolute and uniform convergence on \(\R\).
\end{theorem}

\begin{proof}
The Peter--Weyl decomposition gives
\[
  L^2(Y)
  =
  \widehat\bigoplus_{\ell\ge0}V_\ell\otimes V_\ell^* .
\]
Right translation acts only on the multiplicity factor, hence on the \(\ell\)-th block
\[
  \widetilde V_g^t
  =
  e^{\ii t/2}\Id_{V_\ell}\otimes\pi_\ell^*(g_t).
\]
Taking the ordinary trace of the finite-dimensional \(\ell\)-block gives
\[
  e^{\ii t/2}(2\ell+1)\Tr_{V_\ell^*}(\pi_\ell^*(g_t))
  =
  e^{\ii t/2}(2\ell+1)\chi_\ell(g_t),
\]
which proves \eqref{eq:C-full-trace-Abel}.  The Abel factor \(\rho^\ell\) gives an absolutely convergent smooth function for \(0<\rho<1\).  Since the coefficients have at most polynomial growth, the limit as \(\rho\uparrow1\) exists in \(\mathcal D'(\R)\).  Formula \eqref{eq:C-character} is the standard character formula for the \((2\ell+1)\)-dimensional representation of \(\SO(3)\).

On the multiplicity factor the Hardy projectors are the rank-one projectors
\[
  P_{h,\ell}=\ket{\widehat u_\ell}\bra{\widehat u_\ell},
  \qquad
  P_{\tau,\ell}=\ket{\widehat w_\ell(\tau)}\bra{\widehat w_\ell(\tau)} .
\]
Thus
\[
  \Tr_{V_\ell\otimes V_\ell^*}
  (\Pi_h\widetilde V_g^t\Pi_h)
  =
  e^{\ii t/2}(2\ell+1)
  \ip{\widehat u_\ell}{\pi_\ell^*(g_t)\widehat u_\ell}.
\]
With \(\widehat u_\ell=z_0^{2\ell}\) and
\[
  g_t\cdot z_0=
  \cos(t/2)z_0+
  \sin(t/2)z_1,
\]
orthogonality of monomials gives
\[
  \ip{\widehat u_\ell}{\pi_\ell^*(g_t)\widehat u_\ell}
  =\cos^{2\ell}(t/2).
\]
This proves \eqref{eq:C-Hopf-trace-Abel}.  The closed form \eqref{eq:C-Hopf-closed} follows from
\[
  \sum_{\ell=0}^{\infty}(2\ell+1)x^\ell
  =\frac{1+x}{(1-x)^2},
  \qquad |x|<1,
\]
with \(x=\cos^2(t/2)\).

The same rank-one computation with \(P_{\tau,\ell}\) gives \eqref{eq:C-Grauert-trace-Abel} and \eqref{eq:C-a-tau-def}.  Since \(\pi_\ell^*(g_t)\) is unitary, \(\abs{a_{\tau,\ell}(t)}\le1\), so the Abel sums again define distributions.

For the mixed trace, the multiplicity-factor trace is
\[
  \Tr_{V_\ell^*}(P_{h,\ell}\pi_\ell^*(g_t)P_{\tau,\ell})
  =
  \ip{\widehat w_\ell(\tau)}{\widehat u_\ell}
  \ip{\widehat u_\ell}{\pi_\ell^*(g_t)\widehat w_\ell(\tau)}.
\]
With the phase convention above, the first factor is \(c_\ell(\tau)\).  For the real Grauert vector,
\[
  w_\ell(\tau)=\bigl(z_0^2+z_1^2+2\alpha z_0z_1\bigr)^\ell,
  \qquad \alpha=\Coth\tau .
\]
Using
\[
  g_t\cdot(z_0,z_1)
  =
  \bigl(\cos(t/2)z_0+\sin(t/2)z_1,
        -\sin(t/2)z_0+\cos(t/2)z_1\bigr),
\]
the coefficient of \(z_0^{2\ell}\) in \(g_t\cdot w_\ell(\tau)\) is
\[
  \bigl(\cos^2(t/2)+\sin^2(t/2)-2\alpha\cos(t/2)\sin(t/2)\bigr)^\ell
  =
  \bigl(1-\alpha\sin t\bigr)^\ell .
\]
Because \(\widehat w_\ell=w_\ell/\|w_\ell\|\) and
\(c_\ell(\tau)=\|w_\ell(\tau)\|^{-1}\), this gives
\[
  \ip{\widehat u_\ell}{\pi_\ell^*(g_t)\widehat w_\ell(\tau)}
  =
  c_\ell(\tau)\bigl(1-\alpha\sin t\bigr)^\ell .
\]
Multiplying by the first factor \(c_\ell(\tau)\) yields the coefficient
\(s_\ell(\tau)(1-\alpha\sin t)^\ell\), and hence \eqref{eq:C-cross-trace-real}.

For the opposite mixed trace one similarly obtains
\[
  \Tr_{V_\ell^*}(P_{\tau,\ell}\pi_\ell^*(g_t)P_{h,\ell})
  =
  c_\ell(\tau)
  \ip{\widehat w_\ell(\tau)}{\pi_\ell^*(g_t)\widehat u_\ell}.
\]
Since $\pi_\ell^*(g_t)$ is unitary,
\[
  \ip{\widehat w_\ell(\tau)}{\pi_\ell^*(g_t)\widehat u_\ell}
  =\overline{\ip{\widehat u_\ell}{\pi_\ell^*(g_{-t})\widehat w_\ell(\tau)}},
\]
and replacing $t$ by $-t$ in the computation above gives
\[
  \ip{\widehat u_\ell}{\pi_\ell^*(g_{-t})\widehat w_\ell(\tau)}
  =c_\ell(\tau)\bigl(1+\alpha\sin t\bigr)^\ell.
\]
This number is real: $u_\ell$ and $w_\ell(\tau)$ have real coefficients, the rotation $g_{-t}$ acts on $(z_0,z_1)$ with real matrix entries, and polynomials with real coefficients pair to real values in the monomial basis.  Taking the conjugate therefore changes nothing, and the coefficient of the opposite mixed trace is $s_\ell(\tau)(1+\alpha\sin t)^\ell$.  This proves \eqref{eq:C-cross-opposite-real}.

It remains only to justify absolute convergence of the mixed series.  From the Legendre asymptotic in the norm calculation,
\[
  s_\ell(\tau)=O_\tau\bigl(r(\tau)^{2\ell}\bigr).
\]
For real \(t\),
\[
  \abs{1-\Coth\tau\sin t}\le1+\Coth\tau,
\]
and
\[
  r(\tau)^2\bigl(1+\Coth\tau\bigr)
  =
  \frac{1-e^{-2\tau}}2
  =r(\tau)<1 .
\]
Therefore both mixed series are dominated by convergent geometric series, uniformly in \(t\in\R\).
\end{proof}

\begin{theorem}[Coefficients at the geodesic periods]
\label{thm:period-symbols}
Let $t_0=2\pi k$ with $k\in\mathbb Z$.  The level-$\ell$ coefficients of the two diagonal traces and either mixed trace satisfy
\begin{align}
  \ip{\widehat u_\ell}{\pi_\ell^*(g_{t_0})\widehat u_\ell}&=1,\label{eq:C-hh-period}\\
  \ip{\widehat w_\ell(\tau)}{\pi_\ell^*(g_{t_0})\widehat w_\ell(\tau)}&=1,\label{eq:C-tt-period}\\
  \Tr_{V_\ell^*}(P_{h,\ell}\pi_\ell^*(g_{t_0})P_{\tau,\ell})
  &=s_\ell(\tau).\label{eq:C-cross-period}
\end{align}
Moreover,
\begin{equation}\label{eq:C-s-principal}
  s_\ell(\tau)
  =C_{\BF}(\tau)^2r(\tau)^{2\ell}
   \left(1+O_\tau(\ell^{-1})\right),
  \qquad \ell\to\infty,
\end{equation}
with $r(\tau)$ and $C_{\BF}(\tau)$ as in \eqref{eq:intro-factorization}.  Thus the two diagonal coefficients are identically one, while the mixed coefficient satisfies
\[
  r(\tau)^{-2\ell}s_\ell(\tau)\longrightarrow C_{\BF}(\tau)^2.
\]
This theorem concerns the coefficients of the four traces with Hardy projectors.  It does not identify their sum with the full flat trace, even modulo a smooth function.
\end{theorem}

\begin{proof}
Since $g_{2\pi k}=e$ in $\SO(3)$, the operator $\pi_\ell^*(g_{t_0})$ is the identity on $V_\ell^*$.  This proves the two diagonal identities.  For the mixed coefficient,
\[
  \Tr_{V_\ell^*}(P_{h,\ell}P_{\tau,\ell})
  =\ip{\widehat w_\ell(\tau)}{\widehat u_\ell}
   \ip{\widehat u_\ell}{\widehat w_\ell(\tau)}
  =s_\ell(\tau).
\]
The large-$\ell$ overlap asymptotic gives
\[
  c_\ell(\tau)
  =C_{\BF}(\tau)r(\tau)^\ell
   \left(1+O_\tau(\ell^{-1})\right).
\]
Squaring proves \eqref{eq:C-s-principal}.
\end{proof}

\begin{theorem}[Difference from the full flat trace at the periods]
\label{thm:trace-difference}
Define the sum of the four traces with Hardy projectors
\begin{equation}\label{eq:C-four-trace-sum}
  \mathcal H_\tau(t)
  =
  T_{hh}(t)+T_{\tau\tau}(t)+T_{h\tau}(t)+T_{\tau h}(t)
\end{equation}
and the difference distribution
\begin{equation}\label{eq:C-trace-difference}
  \mathcal D_\tau(t)
  =
  \Tr^{\flat}(\widetilde V_g^t)-\mathcal H_\tau(t).
\end{equation}
Then \(\mathcal D_\tau\) is not a smooth function.  More precisely, for \(t_0=2\pi k\) and for the Abel regularization obtained by inserting the same factor \(\rho^\ell\) in every spectral series, one has
\begin{equation}\label{eq:C-trace-difference-Abel-exact}
  \mathcal D_{\tau,\rho}(t_0)
  =
  (-1)^k
  \sum_{\ell=0}^{\infty}\rho^\ell(2\ell+1)
  \Bigl((2\ell+1)-2-2s_\ell(\tau)\Bigr).
\end{equation}
As \(\rho\uparrow1\),
\begin{equation}\label{eq:C-trace-difference-asymptotic}
  \mathcal D_{\tau,\rho}(t_0)
  =
  (-1)^k
  \left(
    \frac{8}{(1-\rho)^3}
    -\frac{12}{(1-\rho)^2}
    +\frac{3}{1-\rho}
    +O_\tau(1)
  \right).
\end{equation}
In particular, the full flat trace is not equal to the sum of the four traces with Hardy projectors modulo a smooth function near any geodesic period.
\end{theorem}

\begin{proof}
At a period \(t_0=2\pi k\), the phase factor is
\(e^{\ii t_0/2}=(-1)^k\), and \(g_{t_0}=e\).  Therefore the \(\ell\)-th full flat-trace coefficient is
\[
  (-1)^k(2\ell+1)\chi_\ell(e)
  =
  (-1)^k(2\ell+1)^2 .
\]
The two diagonal Hardy coefficients are each
\[
  (-1)^k(2\ell+1),
\]
and the two mixed coefficients are each
\[
  (-1)^k(2\ell+1)s_\ell(\tau).
\]
Subtracting gives the exact Abel-regularized identity \eqref{eq:C-trace-difference-Abel-exact}.

The elementary sums
\begin{align*}
  \sum_{\ell=0}^{\infty}\rho^\ell(2\ell+1)^2
  &=
  \frac{1+6\rho+\rho^2}{(1-\rho)^3},\\
  2\sum_{\ell=0}^{\infty}\rho^\ell(2\ell+1)
  &=
  \frac{2(1+\rho)}{(1-\rho)^2}
\end{align*}
are obtained from
\(\sum\rho^\ell=(1-\rho)^{-1}\) by applying \(\rho\partial_\rho\).  Since
\(s_\ell(\tau)=O_\tau(r(\tau)^{2\ell})\) with \(0<r(\tau)<1\), the mixed Abel sum
\[
  2\sum_{\ell=0}^{\infty}\rho^\ell(2\ell+1)s_\ell(\tau)
\]
converges to a finite limit as \(\rho\uparrow1\).  Consequently
\begin{align*}
  \mathcal D_{\tau,\rho}(t_0)
  &=
  (-1)^k
  \left[
  \frac{1+6\rho+\rho^2}{(1-\rho)^3}
  -
  \frac{2(1+\rho)}{(1-\rho)^2}
  +O_\tau(1)
  \right]\\
  &=
  (-1)^k
  \left[
  \frac{-1+6\rho+3\rho^2}{(1-\rho)^3}
  +O_\tau(1)
  \right].
\end{align*}
Writing \(\delta=1-\rho\), the numerator is $8-12\delta+3\delta^2$,
which proves \eqref{eq:C-trace-difference-asymptotic}.

The factor $\rho^\ell=e^{-\varepsilon\ell}$ is the Abel regularization in the representation index $\ell$.  On every open set where the limiting distribution is smooth, these regularizations converge in $C^\infty$ as $\varepsilon\downarrow0$.  The cubic divergence in \eqref{eq:C-trace-difference-asymptotic} therefore rules out smoothness of $\mathcal D_\tau$ near $t_0$.
\end{proof}

\section{Fredholm determinant of \texorpdfstring{$\mathsf T_\tau$}{T tau}}
\label{sec:fredholm-determinant}
The spectrum in Proposition~\ref{prop:T-spectrum} is summable, so the Fredholm determinant
\begin{equation}
\label{eq:fredholm-determinant}
  D_{\mathsf T}(z,\tau)=\det(I-z\mathsf T_\tau)
\end{equation}
is defined for every $z\in\C$.

\begin{theorem}[Product formula, order zero, and genus zero]
\label{thm:det-order-zero}
For each $\tau>0$, the function $D_{\mathsf T}(\,\cdot\,,\tau)$ is entire and
\begin{equation}
\label{eq:det-product}
  D_{\mathsf T}(z,\tau)
  =\prod_{\ell=0}^{\infty}\bigl(1-zs_\ell(\tau)\bigr)^{2\ell+1}.
\end{equation}
It has order zero.  More precisely, for every $\varepsilon>0$ there is a constant $C_{\tau,\varepsilon}$ such that
\begin{equation}
\label{eq:det-growth-order-zero}
  \log\abs{D_{\mathsf T}(z,\tau)}
  \le C_{\tau,\varepsilon}(1+\abs z)^\varepsilon.
\end{equation}
The product in \eqref{eq:det-product} is the canonical Hadamard product of genus zero.
\end{theorem}

\begin{proof}
The exponential estimate $s_\ell(\tau)=O_\tau(r(\tau)^{2\ell})$ gives
\begin{equation}
\label{eq:all-eps-summability}
  \sum_{\ell\ge0}(2\ell+1)s_\ell(\tau)^\varepsilon<\infty
  \qquad \text{for every }\varepsilon>0.
\end{equation}
The case $\varepsilon=1$ proves trace-class summability and gives the Fredholm product \eqref{eq:det-product}.  For $0<\varepsilon\le1$, the inequality $\log(1+x)\le C_\varepsilon x^\varepsilon$ yields
\[
\begin{aligned}
  \log\abs{D_{\mathsf T}(z,\tau)}
  &\le\sum_{\ell\ge0}(2\ell+1)\log\bigl(1+\abs z\,s_\ell(\tau)\bigr)\\
  &\le C_\varepsilon\abs z^\varepsilon
      \sum_{\ell\ge0}(2\ell+1)s_\ell(\tau)^\varepsilon.
\end{aligned}
\]
This proves \eqref{eq:det-growth-order-zero} and hence order zero.  The zeros are $s_\ell(\tau)^{-1}$ with multiplicity $2\ell+1$.  Equation~\eqref{eq:all-eps-summability} shows that their exponent of convergence is zero, so no higher genus factors are required.  Since $D_{\mathsf T}(0,\tau)=1$, the canonical product is exactly \eqref{eq:det-product}.
\end{proof}

\begin{proposition}[Logarithmic derivative]
\label{prop:log-derivative}
For every $\tau>0$,
\begin{equation}
\label{eq:log-derivative}
  \partial_z\log D_{\mathsf T}(z,\tau)
  =-\sum_{\ell=0}^{\infty}(2\ell+1)
    \frac{s_\ell(\tau)}{1-zs_\ell(\tau)}.
\end{equation}
For $\abs z<\norm{\mathsf T_\tau}^{-1}$,
\begin{equation}
\label{eq:log-derivative-zeta-series}
  \partial_z\log D_{\mathsf T}(z,\tau)
  =-\sum_{n=1}^{\infty}z^{n-1}Z_{\mathsf T}(n,\tau),
\end{equation}
where
\[
  Z_{\mathsf T}(n,\tau)=\sum_{\ell=0}^{\infty}(2\ell+1)s_\ell(\tau)^n.
\]
\end{proposition}

\begin{proof}
Termwise differentiation of \eqref{eq:det-product} gives \eqref{eq:log-derivative}.  Expanding $s/(1-zs)=\sum_{n\ge1}z^{n-1}s^n$ in the stated disk and interchanging the absolutely convergent sums gives \eqref{eq:log-derivative-zeta-series}.
\end{proof}

\section{Zeta function associated with \texorpdfstring{$\mathsf T_\tau$}{T tau}}
\label{sec:T-zeta}

\subsection{Definition}

\begin{definition}
\label{def:T-zeta}
For $\tau>0$ and $\Re(s)>0$ we define the \emph{zeta function associated with $\mathsf T_\tau$} by
\begin{equation}
\label{eq:T-zeta}
Z_{\mathsf{T}}(s,\tau)\;:=\;\sum_{\ell=0}^\infty (2\ell+1)\,s_\ell(\tau)^s.
\end{equation}
\end{definition}

Here $s_\ell(\tau)^s=\exp\bigl(s\log s_\ell(\tau)\bigr)$, with the real logarithm.  One has $s_0(\tau)=1$ and $0<s_\ell(\tau)<1$ for $\ell\ge1$.  Since $s_\ell(\tau)$ decays exponentially in $\ell$ by Corollary~\ref{cor:asymptotics}, the series \eqref{eq:T-zeta} converges absolutely for $\Re(s)>0$.
The large-$\ell$ expansion below gives a finite polylogarithmic expansion of the zeta series near the boundary of this half-plane.

\subsection{Eigenvalue asymptotics}

We first record the precise form of the large-$\ell$ expansion of $s_\ell(\tau)$.
Corollary~\ref{cor:asymptotics} gives the leading exponential scale.
A full expansion follows from Szeg\H{o}'s asymptotic for Legendre polynomials
$\Leg_\ell(\cosh 2\tau)$ and Stirling's formula for the central binomial coefficient.

\begin{proposition}[Complete large-$\ell$ expansion]
\label{prop:sell-expansion}
Fix $\tau>0$.
There exist smooth coefficients $a_j(\tau)$, $j\ge 1$, such that as $\ell\to\infty$,
\begin{equation}
\label{eq:sell-expansion}
s_\ell(\tau)
\;=\;
C_0(\tau)\,r(\tau)^{2\ell}
\Bigl(1+\frac{a_1(\tau)}{\ell}+\frac{a_2(\tau)}{\ell^2}+\cdots+\frac{a_N(\tau)}{\ell^N}
+\mathcal{O}_{N,\tau}(\ell^{-N-1})\Bigr),
\end{equation}
where $C_0(\tau)$ and $r(\tau)$ are as in Section~\ref{subsec:notation}, and
\begin{align}
\label{eq:a12-explicit}
a_1(\tau)
&=\frac{1-\Coth(2\tau)}{8}
=-\frac{1}{4(e^{4\tau}-1)},\\
a_2(\tau)
&=-\frac{(\Coth(2\tau)-1)(7\Coth(2\tau)+1)}{128}
=-\frac{4e^{4\tau}+3}{32(e^{4\tau}-1)^2}.
\end{align}
\end{proposition}

\begin{proof}
We start from the exact identity
\begin{equation}
\label{eq:sell-exact-proof}
s_\ell(\tau)
=\frac{\binom{2\ell}{\ell}}{4^\ell}
  \,\frac{\sinh^{2\ell}\!\tau}{\Leg_\ell(\cosh 2\tau)}.
\end{equation}
The first two coefficients follow from Stirling's formula and the endpoint Laplace expansion of the Legendre integral.  The full Szeg\H{o} asymptotic series gives the expansion to arbitrary order.

Stirling's formula gives
\begin{equation}
\label{eq:central-binomial-refined}
\frac{\binom{2\ell}{\ell}}{4^\ell}
=
\frac{1}{\sqrt{\pi\ell}}
\left(1-\frac{1}{8\ell}+\frac{1}{128\ell^2}+\mathcal{O}(\ell^{-3})\right).
\end{equation}
To obtain the first two Legendre correction terms directly, set $\eta=2\tau$ and use the Laplace
integral representation
\[
\Leg_\ell(\cosh\eta)=\frac{1}{\pi}\int_0^\pi \bigl(\cosh\eta+\sinh\eta\cos\theta\bigr)^\ell\,d\theta
=\frac{1}{\pi}\int_0^\pi e^{\ell\varphi_\eta(\theta)}\,d\theta,
\]
where $\varphi_\eta(\theta):=\log(\cosh\eta+\sinh\eta\cos\theta)$.
Expanding at $\theta=0$ gives
\[
\varphi_\eta(\theta)
=
\eta-a_\eta\theta^2+c_\eta\theta^4+d_\eta\theta^6+\mathcal{O}_\eta(\theta^8),
\]
where
\[
a_\eta=\frac{1-e^{-2\eta}}{4},
\qquad
c_\eta=-\frac{(1-e^{-2\eta})(1-3e^{-2\eta})}{96},
\qquad
d_\eta=-\frac{(1-e^{-2\eta})(2-15e^{-2\eta}+15e^{-4\eta})}{2880}.
\]
Choose $\delta\in(0,\pi)$ so small that $\theta=0$ is the only maximum of $\varphi_\eta$ on $[0,\delta]$ and
\[
  \varphi_\eta(\theta)\le \eta-\frac{a_\eta}{2}\theta^2
  \qquad (0\le\theta\le\delta).
\]
The integral over $[\delta,\pi]$ is $O_\eta(e^{(\eta-\varepsilon_\eta)\ell})$ for some $\varepsilon_\eta>0$.  On $[0,\delta]$, set $\theta=u/\sqrt{\ell a_\eta}$.  The standard endpoint Laplace expansion gives
\[
\Leg_\ell(\cosh\eta)
=
\frac{e^{\ell\eta}}{\pi\sqrt{\ell a_\eta}}
\int_0^\infty e^{-u^2}
\left(
1+\frac{\beta_\eta u^4}{\ell}
+\frac{\gamma_\eta u^6+\frac12\beta_\eta^2u^8}{\ell^2}
\right)\,du
+O_\eta\!\left(e^{\ell\eta}\ell^{-7/2}\right),
\]
where
\[
\beta_\eta=\frac{c_\eta}{a_\eta^2}=\frac{\Coth\eta-2}{6},
\qquad
\gamma_\eta=\frac{d_\eta}{a_\eta^3}.
\]
The error estimate follows by applying the Taylor expansion on a shrinking neighborhood of the origin and using the Gaussian bound above on its complement.
Using
\[
\int_0^\infty e^{-u^2}\,du=\frac{\sqrt{\pi}}{2},\qquad
\int_0^\infty u^4e^{-u^2}\,du=\frac{3\sqrt{\pi}}{8},
\]
\[
\int_0^\infty u^6e^{-u^2}\,du=\frac{15\sqrt{\pi}}{16},\qquad
\int_0^\infty u^8e^{-u^2}\,du=\frac{105\sqrt{\pi}}{32},
\]
we find
\[
\Leg_\ell(\cosh\eta)
=
\frac{e^{\ell\eta}}{\sqrt{4\pi\ell a_\eta}}
\left(
1+\frac{3\beta_\eta}{4\ell}
+\frac{\frac{15}{8}\gamma_\eta+\frac{105}{32}\beta_\eta^2}{\ell^2}
+\mathcal{O}_\eta(\ell^{-3})
\right).
\]
Since $4a_\eta=1-e^{-2\eta}=2e^{-\eta}\sinh\eta$ and
\[
\frac{3\beta_\eta}{4}=\frac{\Coth\eta-2}{8},
\qquad
\frac{15}{8}\gamma_\eta+\frac{105}{32}\beta_\eta^2=\frac{(3\Coth\eta-2)^2}{128},
\]
this becomes
\begin{equation}
\label{eq:legendre-refined}
\Leg_\ell(\cosh\eta)
=
\frac{e^{(\ell+1/2)\eta}}{\sqrt{2\pi\ell\sinh\eta}}
\left(
1+\frac{\Coth\eta-2}{8\ell}
+\frac{(3\Coth\eta-2)^2}{128\,\ell^2}
+\mathcal{O}_\eta(\ell^{-3})
\right).
\end{equation}
Substituting \eqref{eq:central-binomial-refined} and \eqref{eq:legendre-refined} with
$\eta=2\tau$ into \eqref{eq:sell-exact-proof} yields
\[
s_\ell(\tau)
=
\sqrt{2\sinh(2\tau)}\,e^{-\tau}
\left(\frac{\sinh\tau}{e^\tau}\right)^{2\ell}
\left(
1+\frac{a_1(\tau)}{\ell}+\frac{a_2(\tau)}{\ell^2}
+\mathcal{O}_\tau(\ell^{-3})
\right),
\]
where
\[
a_1(\tau)
=
-\frac18-\frac{\Coth(2\tau)-2}{8}
=
\frac{1-\Coth(2\tau)}{8},
\]
and
\[
\begin{aligned}
a_2(\tau)
&=
\frac{1}{128}
+\frac{\Coth(2\tau)-2}{64}
+\frac{(\Coth(2\tau)-2)^2}{64}
-\frac{(3\Coth(2\tau)-2)^2}{128}\\
&=-\frac{(\Coth(2\tau)-1)(7\Coth(2\tau)+1)}{128}.
\end{aligned}
\]
Because $\sinh\tau/e^\tau=(1-e^{-2\tau})/2=r(\tau)$, this gives
\eqref{eq:a12-explicit}. The complete expansion \eqref{eq:sell-expansion} then follows by
combining \eqref{eq:central-binomial-refined} with the full Legendre asymptotic series of
Szeg\H{o} \cite[Ch.~VIII, Thm.~8.21.2]{Szego}.  See also \cite[\S14.15]{DLMF}.
\end{proof}

\subsection{Finite polylogarithmic expansion}

\begin{theorem}[Finite polylogarithmic expansion]
\label{thm:zeta-polylog-expansion}
For each fixed $\tau>0$, put $r=r(\tau)$, $C_0=C_0(\tau)$, and $q(s)=r^{2s}$.  For every $N\ge2$ there are polynomials $p_{j,\tau}(s)$, $1\le j\le N$, whose coefficients depend smoothly on $\tau$ and for which $p_{j,\tau}(0)=0$, such that for $\Re(s)>0$
\begin{equation}
\label{eq:zeta-polylog-split}
Z_{\mathsf T}(s,\tau)
=
C_0^s\frac{1+q(s)}{(1-q(s))^2}
+C_0^s\sum_{j=1}^{N}p_{j,\tau}(s)\Bigl(2\Li_{j-1}(q(s))+
\Li_{j}(q(s))\Bigr)
+\mathcal R_N(s,\tau),
\end{equation}
where $\mathcal R_N$ is holomorphic in $\Re(s)>0$.  On every compact set $K\Subset\{\Re(s)>0\}$, the defining remainder series and all of its $s$-derivatives converge uniformly.  The displayed finite sum extends to the universal cover of
\begin{equation}
\label{eq:zeta-pole-lattice}
\C\setminus\mathcal{P}_\tau,
\qquad
\mathcal{P}_\tau:=\{s\in\C:q(s)=1\}
=\left\{\frac{\pi i k}{\log r(\tau)}:k\in\mathbb Z\right\},
\end{equation}
and has polylogarithmic singularities along $\mathcal P_\tau$.  The first term on the right-hand side of \eqref{eq:zeta-polylog-split} has a
double pole at each point of $\mathcal P_\tau$, while the remaining displayed terms have at most simple
poles together with the logarithmic terms prescribed by Jonqui\`ere's expansion of the
polylogarithm.  At $s=0$ the polar part of this first term is
\begin{equation}
\label{eq:zeta-laurent-at-0}
C_0^s\frac{1+q(s)}{(1-q(s))^2}
=
\frac{1}{2(\log r(\tau))^2}\,\frac{1}{s^2}
+\frac{\log C_0(\tau)-\log r(\tau)}{2(\log r(\tau))^2}\,\frac{1}{s}
+\mathcal O(1).
\end{equation}
Both coefficients vary with $\tau$.
\end{theorem}

\begin{remark}[Scope of the continuation]
\label{rem:zeta-polylog-not-meromorphic}
The functions $\Li_m(q)$ with positive integer $m$ have logarithmic singularities at $q=1$.  Thus the truncated expression in \eqref{eq:zeta-polylog-split} is generally not meromorphic.  The theorem continues only the displayed finite expression.  It gives no continuation of the remainder $\mathcal R_N$ or of the full zeta function across $\Re s=0$.
\end{remark}

\begin{proof}
We write $r=r(\tau)$ and $C_0=C_0(\tau)$ for brevity.

\smallskip\noindent\emph{Reduction to polylogarithms.}
Fix $N\ge2$ and use \eqref{eq:sell-expansion} for $\ell\ge1$.  The $\ell=0$ term in $Z_{\mathsf T}$ is $1$.  The first term below contributes $C_0^s$ at $\ell=0$, and the entire correction $1-C_0^s$ is absorbed into the holomorphic remainder $\mathcal R_N$.  For $\ell\ge1$ and $\Re(s)>0$,
\[
s_\ell(\tau)^s
=\bigl(C_0 r^{2\ell}\bigr)^s
\Bigl(1+\frac{a_1}{\ell}+\cdots+\frac{a_N}{\ell^N}+\mathcal O_{N,\tau}(\ell^{-N-1})\Bigr)^s.
\]
Taylor expansion of $(1+u)^s$ at $u=0$, applied to the finite polynomial in $\ell^{-1}$,
gives
\begin{equation}
\label{eq:power-expansion}
\Bigl(1+\frac{a_1}{\ell}+\cdots+\frac{a_N}{\ell^N}+\mathcal O_{N,\tau}(\ell^{-N-1})\Bigr)^s
=
1+\sum_{j=1}^{N}\frac{p_{j,\tau}(s)}{\ell^j}+\mathcal O_{N,\tau,K}(\ell^{-N-1})
\end{equation}
uniformly for $s$ in compact subsets $K\Subset\C$.  Each $p_{j,\tau}$ is a polynomial in $s$ with
coefficients depending smoothly on $\tau$, and $p_{j,\tau}(0)=0$ because the left hand side of
\eqref{eq:power-expansion} equals $1$ at $s=0$.
Substitution into \eqref{eq:T-zeta} gives, for $\Re(s)>0$,
\[
Z_{\mathsf T}(s,\tau)=
C_0^s\sum_{\ell=0}^{\infty}(2\ell+1)q(s)^\ell
+C_0^s\sum_{j=1}^{N}p_{j,\tau}(s)
\sum_{\ell=1}^{\infty}(2\ell+1)\ell^{-j}q(s)^\ell
+\mathcal R_N(s,\tau).
\]
The first sum is
\begin{equation}
\label{eq:geometric-sum}
\sum_{\ell=0}^{\infty}(2\ell+1)q^\ell=\frac{1+q}{(1-q)^2}.
\end{equation}
For the remaining sums one uses the defining series of the polylogarithm,
\[
\Li_m(q)=\sum_{\ell=1}^{\infty}\ell^{-m}q^\ell,
\qquad |q|<1,
\]
with the convention $\Li_0(q)=q/(1-q)$.  Since
$(2\ell+1)\ell^{-j}=2\ell^{1-j}+\ell^{-j}$, one obtains the exact identity
\begin{equation}
\label{eq:Fj-polylog}
F_j(q):=\sum_{\ell=1}^{\infty}(2\ell+1)\ell^{-j}q^\ell
=2\Li_{j-1}(q)+\Li_{j}(q),
\qquad |q|<1.
\end{equation}
This proves \eqref{eq:zeta-polylog-split} in the original half-plane of convergence.

\smallskip\noindent\emph{Remainder estimate in the convergent half-plane.}
Let $K\Subset\{\Re(s)>0\}$ and put $\sigma_K=\inf_{s\in K}\Re(s)>0$.  The remainder in \eqref{eq:power-expansion} satisfies
\[
\bigl|(2\ell+1)(C_0r^{2\ell})^s\mathcal O_{N,\tau,K}(\ell^{-N-1})\bigr|
\le C_K(2\ell+1)\ell^{-N-1}r^{2\ell\sigma_K}.
\]
This majorant is summable.  Each differentiation in $s$ introduces at most a polynomial factor in $\ell$, while $r^{2\ell\sigma_K}$ still gives exponential decay.  The remainder series and all of its $s$-derivatives therefore converge uniformly on $K$, and $\mathcal R_N$ is holomorphic in $\Re(s)>0$.

\smallskip\noindent\emph{Singularities of the displayed finite sum.}
The displayed finite sum in \eqref{eq:zeta-polylog-split} is a composition of entire functions of $s$
with the classical continuations of $\Li_m(q)$.  The only possible singularities
occur when $q(s)=1$, namely at \eqref{eq:zeta-pole-lattice}.  The leading term \eqref{eq:geometric-sum} has a double pole there.  The term with $j=1$ contains
$2\Li_0(q)$ and hence can contribute a simple pole.  All terms with
$\Li_m$, $m\ge1$, are governed near $q=e^{\mu}=1$ by Jonqui\`ere's formula
\begin{equation}
\label{eq:jonquiere-local}
\Li_m(e^{\mu})
=
\frac{\mu^{m-1}}{(m-1)!}\bigl(H_{m-1}-\log(-\mu)\bigr)
+\sum_{\substack{k\ge0\\ k\ne m-1}}\frac{\zeta(m-k)}{k!}\mu^k,
\qquad m\in\N,
\end{equation}
valid in the slit neighborhood $|\mu|<\epsilon$, $\mu\notin[0,\infty)$, with $H_0=0$ and the principal branch of $\log(-\mu)$.  On the branch obtained by continuation from $|q|<1$, one has $\Re\mu<0$.  This is the branch reached from the half-plane where the defining series converges.  Thus the remaining terms have no singularity worse than a simple pole from
$\Li_0$ together with logarithmic terms.

\smallskip\noindent\emph{Polar part at the origin.}
Let $L=\log r<0$ and put $x=2Ls$.  Since $q(s)=e^x$, the elementary Laurent expansion
\[
  \frac{1+e^x}{(1-e^x)^2}=\frac{2}{x^2}-\frac{1}{x}+O(1)
  \qquad (x\to0)
\]
gives
\[
  \frac{1+q(s)}{(1-q(s))^2}
  =\frac{1}{2L^2s^2}-\frac{1}{2Ls}+O(1).
\]
Multiplying by $C_0^s=1+s\log C_0+O(s^2)$ proves \eqref{eq:zeta-laurent-at-0}.

It remains to check that the remaining terms do not alter this polar part.  Since $p_{1,\tau}(0)=0$, write $p_{1,\tau}(s)=s\widetilde p_{1,\tau}(s)$.  Also
\[
  \Li_0(q(s))=\frac{q(s)}{1-q(s)}=-\frac{1}{2Ls}+O(1),
\]
so $p_{1,\tau}(s)\,2\Li_0(q(s))$ is holomorphic at $s=0$.  The terms involving $\Li_m$ with $m\ge1$ contribute logarithmic terms multiplied by functions that vanish at $s=0$.  They therefore add no pole at the origin.  Thus the coefficients displayed in \eqref{eq:zeta-laurent-at-0} are exactly the polar coefficients of the first term in \eqref{eq:zeta-polylog-split}.
\end{proof}

\begin{corollary}[Monotonicity of the $s^{-1}$ coefficient]
\label{cor:zeta-residue-monotonicity}
Let
\[
\mathcal R(\tau):=\operatorname*{Res}_{s=0}\left(C_0(\tau)^s\frac{1+r(\tau)^{2s}}{(1-r(\tau)^{2s})^2}\right)
=\frac{\log C_0(\tau)-\log r(\tau)}{2(\log r(\tau))^2}.
\]
Then $\mathcal R$ is strictly increasing on $(0,\infty)$.  Hence the coefficient of $s^{-1}$ in the first term of \eqref{eq:zeta-polylog-split} determines $\tau$ within this family.
\end{corollary}
\begin{proof}
Set $x=e^{-2\tau}$ and $A=\log\frac{2}{1-x}>0$.  Since $C_0(\tau)=\sqrt{1-x^2}$ and $r(\tau)=(1-x)/2$,
\[
\mathcal R(\tau)=\frac{B(x)}{2A(x)^2},
\qquad
B(x)=\log2+\frac12\log\frac{1+x}{1-x}.
\]
A direct differentiation, using $dx/d\tau=-2x$, gives
\[
\frac{d\mathcal R}{d\tau}
=\frac{x}{A^3(1-x)}\left(2B-\frac{A}{1+x}\right)
=\frac{x}{A^3(1-x)}\left(\frac{xA}{1+x}+\log(2(1+x))\right)>0.
\]
Thus $\mathcal R$ is strictly increasing.
\end{proof}

\section{Bargmann--Fock calculation of the leading coefficient}
\label{sec:BF-leading}
To compute the coefficient $C_{\BF}(\tau)$ in \eqref{eq:intro-factorization}, we linearize the two contact structures at the identity and identify their horizontal planes with the standard Heisenberg plane.  The linearization produces a symplectic map $S_\tau$ between the two horizontal planes.  The matrix coefficient of the associated metaplectic operator on the Gaussian ground state equals $C_{\BF}(\tau)$.  No global Fourier integral description of $\Pi_h\Pi_\tau$ is used.

\subsection{The two real contact cones}
\begin{proposition}[Distinct contact cones]
\label{prop:no-common-cone}
Normalize the Hopf and Grauert contact forms at the identity by
\[
  \alpha_h(E_3)=1,
  \quad \alpha_h(E_1)=\alpha_h(E_2)=0,
\]
\[
  \alpha_\tau(E_2)=1,
  \quad \alpha_\tau(E_1)=\alpha_\tau(E_3)=0.
\]
Then
\[
  H_h=\ker\alpha_h=\Span\{E_1,E_2\},
  \qquad
  H_\tau=\ker\alpha_\tau=\Span\{E_1,E_3\}.
\]
The positive cones
\[
  \Sigma_h=\{(y,r\alpha_h(y)):r>0\},
  \qquad
  \Sigma_\tau=\{(y,s\alpha_\tau(y)):s>0\}
\]
are disjoint in $T^*Y\setminus0$.
\end{proposition}

\begin{proof}
If $r\alpha_h=s\alpha_\tau$ with $r,s>0$, evaluation on $E_3$ gives $r=0$, a contradiction.  Left invariance gives the same conclusion at every point.
\end{proof}

The horizontal spaces share the line $\R E_1$, but the positive covector cones do not meet.  The composition calculus of \cite{BdMG} presupposes that the two Szeg\H{o} projectors are associated with the same symplectic cone.  Since $\Sigma_h\cap\Sigma_\tau=\emptyset$, that calculus does not apply to the product $\Pi_h\Pi_\tau$.  The leading coefficient is instead computed below from the linear model at a single point.  

\subsection{Linear model at the identity}
\begin{lemma}[Differential of the complexified geodesic-flow map]
\label{lem:complex-flow-scales}
Let
\[
  z_\tau=\ii\sinh\tau\,e_1+\cosh\tau\,e_3,
  \qquad
  \Phi_\tau(g)=gz_\tau.
\]
At the identity,
\begin{equation}
\label{eq:dPhi}
  d\Phi_\tau(E_1)=-\cosh\tau\,e_2,
  \qquad
  d\Phi_\tau(E_3)=\ii\sinh\tau\,e_2.
\end{equation}
Thus the two directions of the Grauert contact plane map to the real and imaginary axes of $\C e_2$ with scale factors $\cosh\tau$ and $\sinh\tau$.
\end{lemma}

\begin{proof}
The rotation matrices satisfy $E_1e_1=0$, $E_1e_3=-e_2$, $E_3e_1=e_2$, and $E_3e_3=0$.  Substitution in $d\Phi_\tau(X)=Xz_\tau$ gives \eqref{eq:dPhi}.
\end{proof}

The scale factors in Lemma~\ref{lem:complex-flow-scales} give the following linear map between the Heisenberg contact models.

\begin{proposition}[Contact map between the Heisenberg models]
\label{prop:linear-contact-map}
Put
\[
  c_\tau=\cosh\tau,
  \qquad
  s_\tau=\sinh\tau,
  \qquad
  \lambda_\tau=c_\tau s_\tau.
\]
Choose oriented Heisenberg coordinates $(\theta,q,p)$ for the Hopf structure and $(\varphi,Q,P)$ for the Grauert structure, with the horizontal $Q$- and $P$-axes along the contact directions $E_1$ and $E_3$, and the horizontal $q$- and $p$-axes along the real and imaginary axes of $\C e_2$ that receive them in Lemma~\ref{lem:complex-flow-scales}.  The Heisenberg contact forms are
\[
  \alpha_h^0=d\theta+\frac12(q\,dp-p\,dq),
  \qquad
  \alpha_\tau^0=d\varphi+\frac12(Q\,dP-P\,dQ).
\]
Then
\begin{equation}
\label{eq:Ftau}
  F_\tau(\varphi,Q,P)
  =(\lambda_\tau\varphi,c_\tau Q,s_\tau P)
\end{equation}
satisfies
\begin{equation}
\label{eq:contact-conformal}
  F_\tau^*\alpha_h^0=\lambda_\tau\alpha_\tau^0.
\end{equation}
The cone lift
\begin{equation}
\label{eq:chitau}
  \chi_\tau(\varphi,Q,P,\sigma)
  =(\lambda_\tau\varphi,c_\tau Q,s_\tau P,\sigma/\lambda_\tau)
\end{equation}
preserves the Liouville form and is therefore exact symplectic.  After the conformal factor is removed, the horizontal block is
\begin{equation}
\label{eq:S-tau}
  S_\tau
  =\begin{pmatrix}
     \sqrt{\Coth\tau}&0\\[2pt]
     0&\sqrt{\tanh\tau}
   \end{pmatrix}
  \in\Sp(2,\R).
\end{equation}
Set
\begin{equation}
\label{eq:b-tau}
  b_\tau=\frac{\Coth\tau-1}{\Coth\tau+1}=e^{-2\tau}.
\end{equation}
\end{proposition}

\begin{proof}
Pulling back the horizontal term gives
\[
  F_\tau^*\frac12(q\,dp-p\,dq)
  =c_\tau s_\tau\,\frac12(Q\,dP-P\,dQ),
\]
which proves \eqref{eq:contact-conformal}.  If $\rho=\sigma/\lambda_\tau$, then
\[
  \chi_\tau^*(\rho\alpha_h^0)=\sigma\alpha_\tau^0.
\]
Thus $\chi_\tau$ preserves the Liouville form and its differential.  On horizontal variables, $dF_\tau=\diag(c_\tau,s_\tau)$ rescales the symplectic form by $\lambda_\tau$.  Dividing by $\lambda_\tau^{1/2}$ gives \eqref{eq:S-tau}.  Finally, \eqref{eq:b-tau} follows from $\cosh\tau\pm\sinh\tau=e^{\pm\tau}$.
\end{proof}

The matrix $S_\tau$ in \eqref{eq:S-tau} is a linear symplectic map between the two horizontal planes,
\[
  S_\tau=\diag\bigl(\sqrt{\Coth\tau},\,\sqrt{\tanh\tau}\,\bigr)\colon
  \underbrace{(Q,P)\text{-plane}}_{\text{Grauert horizontal plane}}
  \longrightarrow
  \underbrace{(q,p)\text{-plane}}_{\text{Hopf horizontal plane}},
\]
written with respect to oriented orthonormal bases of these two planes.  Both bases were fixed above through the real and imaginary axes of $\C e_2$ in Lemma~\ref{lem:complex-flow-scales}, and these choices are auxiliary.  Replacing the basis of either plane by another oriented orthonormal basis rotates that plane, and so replaces $S_\tau$ by $R\,S_\tau\,R'$ with $R,R'\in\mathrm{SO}(2)=\Sp(2,\R)\cap\mathrm{O}(2)$.  Under the metaplectic representation of Lemma~\ref{lem:metaplectic-gaussian} below, a rotation acts on the normalized Gaussian by a unimodular scalar, because the Gaussian is the ground state of the harmonic oscillator that generates the rotations.  Consequently
\[
  \abs{\ip{\Omega}{\mu(R\,S_\tau\,R')\Omega}}
  =\abs{\ip{\Omega}{\mu(S_\tau)\Omega}},
\]
and the Gaussian matrix coefficient computed below does not depend on the choice of bases.

\subsection{The metaplectic Gaussian-matrix coefficient}
\begin{lemma}
\label{lem:metaplectic-gaussian}
Let $S_a=\diag(a,a^{-1})\in\Sp(2,\R)$ with $a>0$.  For the normalized Gaussian $\Omega(x)=\pi^{-1/4}e^{-x^2/2}$,
\begin{equation}
\label{eq:gaussian-a}
  \abs{\ip{\Omega}{\mu(S_a)\Omega}}
  =\left(\frac{2a}{a^2+1}\right)^{1/2}.
\end{equation}
If $b=(a^2-1)/(a^2+1)$, the same coefficient is $(1-b^2)^{1/4}$.
\end{lemma}

\begin{proof}
Up to a phase, the metaplectic dilation is $(\mu(S_a)f)(x)=a^{-1/2}f(x/a)$.  Hence
\[
\begin{aligned}
  \ip{\Omega}{\mu(S_a)\Omega}
  &=\pi^{-1/2}a^{-1/2}\int_\R
    \exp\left[-\frac12(1+a^{-2})x^2\right]dx\\
  &=\left(\frac{2a}{a^2+1}\right)^{1/2}.
\end{aligned}
\]
The second formula follows from $1-b^2=4a^2/(a^2+1)^2$.
\end{proof}

\begin{proposition}[Bargmann--Fock Gaussian matrix coefficient]
\label{prop:BF-gaussian}
For the matrix $S_\tau$ in \eqref{eq:S-tau},
\begin{equation}
\label{eq:BF-coeff}
  C_{\BF}(\tau)
  :=\abs{\ip{\Omega}{\mu(S_\tau)\Omega}}
  =(1-e^{-4\tau})^{1/4}
  =(2\sinh 2\tau)^{1/4}e^{-\tau/2}.
\end{equation}
In the holomorphic Bargmann--Fock space
\[
  \mathcal F=\left\{f\text{ entire}:\ 
  \norm f_{\mathcal F}^2=\pi^{-1}\int_\C\abs{f(z)}^2e^{-\abs z^2}\,dL(z)<\infty\right\},
\]
the corresponding normalized Gaussian vector is
\begin{equation}
\label{eq:BF-gaussian-vector}
  \psi_\tau(z)=(1-b_\tau^2)^{1/4}
  \exp\left(\frac{b_\tau}{2}z^2\right),
\end{equation}
and $\abs{\ip{1}{\psi_\tau}_{\mathcal F}}=C_{\BF}(\tau)$.
\end{proposition}

\begin{proof}
Apply Lemma~\ref{lem:metaplectic-gaussian} with $a=\sqrt{\Coth\tau}$.  Then
\[
  \frac{a^2-1}{a^2+1}
  =\frac{\Coth\tau-1}{\Coth\tau+1}
  =e^{-2\tau}=b_\tau.
\]
Hence the coefficient is $(1-e^{-4\tau})^{1/4}$.  The identity $1-e^{-4\tau}=2e^{-2\tau}\sinh 2\tau$ gives the last expression in \eqref{eq:BF-coeff}.

The Bargmann transform carries $\Omega$ to the constant function $1$ and conjugates the metaplectic representation on $L^2(\R)$ into its realization on $\mathcal F$.  The matrix coefficient may therefore be computed in $\mathcal F$, where $\mu(S_\tau)\Omega$ becomes, up to phase, the Gaussian vector $\psi_\tau$ of \eqref{eq:BF-gaussian-vector}.  Monomial orthogonality gives
\[
  \left\|\exp\left(\frac b2z^2\right)\right\|_{\mathcal F}^2
  =\sum_{m\ge0}\binom{2m}{m}\left(\frac{b^2}{4}\right)^m
  =(1-b^2)^{-1/2}.
\]
This proves the normalization in \eqref{eq:BF-gaussian-vector}.  Since $\ip{1}{f}_{\mathcal F}=\overline{f(0)}$ under our inner-product convention, evaluation at the origin gives the stated overlap.
\end{proof}

\subsection{Comparison with the exact overlap}
\begin{theorem}[Leading asymptotic of the overlap]
\label{thm:BF-leading-overlap}
For fixed $\tau>0$,
\begin{equation}
\label{eq:finite-overlap-factorization}
  c_\ell(\tau)
  =C_{\BF}(\tau)r(\tau)^\ell
   \bigl(1+O_\tau(\ell^{-1})\bigr),
  \qquad \ell\to\infty.
\end{equation}
Equivalently,
\[
  s_\ell(\tau)
  =C_{\BF}(\tau)^2r(\tau)^{2\ell}
   \bigl(1+O_\tau(\ell^{-1})\bigr).
\]
The coefficient $C_{\BF}(\tau)$ equals the Gaussian matrix coefficient of the metaplectic operator associated with $S_\tau$.
\end{theorem}

\begin{proof}
The asymptotic formula follows from Corollary~\ref{cor:asymptotics}.  Proposition~\ref{prop:BF-gaussian} identifies its leading coefficient with the Gaussian matrix coefficient of $S_\tau$.
\end{proof}

The coefficient $C_{\BF}(\tau)$ comes from the linear symplectic calculation above, whereas the exponential factor is already present in the exact representation-theoretic formula
\[
  r(\tau)^\ell=e^{-\ell A(\tau)},
  \qquad
  A(\tau)=\log\frac{2}{1-e^{-2\tau}}.
\]
Thus the Bargmann--Fock calculation gives the prefactor in \eqref{eq:finite-overlap-factorization}, while the exponential rate comes from the exact overlap formula.

\end{document}